\documentclass[final,leqno,onefignum,onetabnum]{siamltexmm}

\usepackage{todonotes}

\usepackage{amsmath, amssymb}
\usepackage{enumerate}
\usepackage{subcaption}
\usepackage{array}
\usepackage{tabularx}

\newcommand{\mb}[1]{\mathbf{#1}}
\newcommand{\mbb}[1]{\mathbb{#1}}
\newcommand{\mcl}[1]{\mathcal{#1}}

\newcommand{\PP}{\mathbb{P}}
\newcommand{\EE}{\mathbb{E\:}}
\newcommand{\VV}{\mathbf{Var}}
\newcommand{\reals}{\mathbb{R}}
\newcommand{\complex}{\mathbb{C}}
\newcommand{\integers}{\mathbb{N}}

\newcommand{\dd}{\text{d}}

\newcommand{\dequal}{\stackrel{d}{=}}

\DeclareMathOperator*{\argmin}{arg\,min}

\usepackage{color}
\definecolor{myorange}{rgb}{.8,.0,.0}
\newcommand{\myhl}[1]{{{{#1}}}}

\newtheorem{example}{Example}
 \newtheorem{assumption}{Assumption}
\title{Well-posed Bayesian Inverse Problems with Infinitely-Divisible
  and Heavy-Tailed
 Prior Measures
  \thanks{This work was supported in part by the Natural Sciences and
    Engineering Council of Canada.}}

 \author{Bamdad Hosseini \footnotemark[2]}

\begin{document}
\today


\maketitle

\newcommand{\slugmaster}{%

\slugger{}{xxxx}{xx}{x}{x--x}}
\renewcommand{\thefootnote}{\fnsymbol{footnote}}
\footnotetext[2]{Department of mathematics, Simon Fraser University,
  8888 University Drive, Burnaby, BC, V5A 1S6, Canada
  (\email{bhossein@sfu.ca}).}
\renewcommand{\thefootnote}{\arabic{footnote}}

\begin{abstract}
We
present
 a new class of prior measures 
in connection to $\ell_p$ regularization techniques when $p \in(0,1)$
which is based on the generalized Gamma distribution. We show that
the resulting prior measure is heavy-tailed, non-convex and infinitely divisible.
Motivated by this observation we discuss the class of infinitely divisible prior measures
and draw a connection between their
tail behavior and the tail behavior of
their L{\'evy} measures. \myhl{Next, we use the laws of pure jump L{\'e}vy processes in order to 
define new classes of prior measures that are concentrated on the space of functions with bounded 
variation. These priors serve as an alternative to the classic total variation prior and result in well-defined inverse problems.   
 We then study the well-posedness of Bayesian
inverse problems in a general enough setting that encompasses the above mentioned classes of prior measures. 
 We establish that
well-posedness relies on a balance between the growth of the
log-likelihood function and the tail behavior of the prior and apply
our results to special cases such as additive noise
models and linear problems. }
 Finally, we discuss some of the practical aspects
of Bayesian inverse problems such as their consistent approximation and present
three concrete examples of well-posed Bayesian inverse problems with 
heavy-tailed or stochastic process prior measures. 
\end{abstract}



\pagestyle{myheadings}
\thispagestyle{plain}

%
%
%
\section{Introduction}\label{sec:introduction}
\myhl{ Gaussian prior measures are perhaps the most commonly used class of priors in infinite-dimensional 
Bayesian inverse problems. While the Gaussian class is very conveninet to use in both theory and practice, it has serious shortcomings in modelling of certain types of prior knowledge such as sparsity. In this article we introduce some non-Gaussian prior measures that are able to model parameters 
that are compressible or have jump discontinuities. We will discuss our goals in more detail after a brief introduction to the Bayesian framework for solution of inverse problems.}

Consider  the problem of estimating a parameter $u \in
X$ from a set of measurements $y
\in Y$ where both $X$ and $Y$ are Banach spaces and $y$ is associated with
$u$ through a model of the form 
 \begin{equation}
   \label{generic-inverse-problem}
   y = \tilde{\mcl{G}}(u).
 \end{equation} 
$\tilde{\mcl{G}}$ is a generic stochastic mapping that models the
relationship between the parameter and the observed data by taking the
measurement noise into account (be it additive, multiplicative
etc). 
 As an
example, if the measurement noise is additive then we can write
$$
\tilde{\mcl{G}}(u) = \mcl{G}(u) + \eta
$$
where $\mcl{G}: X  \mapsto Y$ is the (deterministic) {\it forward model}
and $\eta$ is the (random) measurement noise which is independent of
$u$. We want to estimate the parameter $u$ given a
 realization of $y$. Since the map $\mcl{G}$ may not be stably invertible
this problem is in
general ill-posed. 

Here we consider the Bayesian
framework for solution of such ill-posed problems. 
Recall the infinite-dimensional version of {\it
   Bayes' rule} \cite{stuart-acta-numerica} which is
understood in the sense of the Radon-Nikodym theorem \cite[Thm.~3.2.2]{bogachev1}:
\begin{equation}
  \label{bayes-rule}
  \frac{ \dd \mu^y }{\dd \mu_0} (u) = \frac{1}{Z(y)} \exp \left( -
    \Phi(u;y) \right) \quad  \text{where} \quad Z(y) = \int_X \exp(-\Phi(u; y)) \dd \mu_0(u).
\end{equation}
Here $\mu_0$ is the {\it prior measure} which reflects our
prior knowledge of the parameter $u$, $\Phi(u;y)$ is the {\it likelihood
  potential} that can be thought of as the negative log of the
density of the data conditioned on the parameter $u$ and $\mu^y$ is the
posterior measure on $u$. The posterior $\mu^y$ is, in essence,
an updated version of the prior $\mu_0$ that is informed by the data
$y$.

The Bayesian approach has attracted a lot of attention
in the last two decades
\cite{somersalo-inverseproblems-review, somersalo,
  stuart-acta-numerica}. Put simply, the unknown parameter $u$ is
modelled as a random variable and our goal is to obtain a probability
distribution $\mu^y$ on $u$ that is informed by the data $y$ and
our prior knowledge about $u$ (modelled by the measure $\mu_0$). We can generate samples from the posterior 
$\mu^y$
 and 
 if this measure is concentrated around the true value of the parameter,
the sample mean or median will be good estimators of  the true value of the parameter.

The Bayesian approach is well-established in the
statistics literature
\cite{calvetti, bernardo} where it is
often applied in the setting where $X,Y$ are
finite-dimensional spaces.
Here we take $X$ to be an infinite-dimensional 
Banach space, motivation by applications where the parameter
$u$ belongs to a function space such as $L^2$ or $BV$
(the space of functions with bounded variation). Such problems arise
when the forward map involves the solution of a partial differential
equation (PDE) or an integral equation such as the examples
 in Section \ref{sec:examples}. 

In practice  we solve these
problems by discretizing the forward model and approximating the
infinite-dimensional problem with a finite dimensional one. An important task is to ensure
 that the finite dimensional approximation to the posterior measure remains consistent with the
infinite-dimensional posterior measure. For example, we require that the finite
dimensional posterior  converges
to the (true) infinite-dimensional measure in the limit when the
discretization is infinitely fine. 
 Ensuring this 
consistency is a delicate task. An example of an inconsistent
discretization of an infinite
dimensional inverse problem was studied in \cite{lassas-can-we-use-tv} where the
authors demonstrated that the total variation prior loses its edge
preserving properties in the limit of fine discretizations. In order to resolve this issue we study the
 infinite-dimensional inverse
problem
before constructing the discrete approximations.

In this article we set out to achieve the following goals: 
\begin{enumerate}[\hspace{1ex} \it{G}1.]
\item \emph{Construct a new class of infinitely divisible prior measures for recovery of compressible parameters in connection to $\ell_p$ regularization techniques when $p \in (0,1)$.}

\item \emph{Present  a systematic study of the class of infinitely divisible prior measures.}
\item \emph{Introduce an alternative to the classic total variation prior using the laws of pure Jump L{\'e}vy process 
that is well-defined in infinite dimensions.}
\item \emph{Present a theory of well-posedness for Bayesian inverse problems that encompasses the 
prior measures introduced under G1--G3.}
\end{enumerate}
Let us motivate some of these goals with an example.
\begin{example}\label{example-1}
Suppose $u \in \reals^n$ and the data 
$y \in \reals^m$ is generated via the model
$$
y = \mb{A} u + \eta, \qquad \eta \sim \mcl{N}(0, \sigma^2\mb{I})
$$
 where  $\mb{A} \in \reals^{n \times m}$, $\sigma >0$ is fixed and
 $\mb{I}$ is the $m \times m$ identity matrix. We wish  to estimate
 $u$ given $y$. 
Here we are taking $X = \reals^n$, $Y =
\reals^m$ and the forward map has the form $\mcl{G}(u) = \mb{A}u$. 
Since $\eta$ has a Gaussian density  we
can write the likelihood potential $\Phi(u;y)$ as:
$$\Phi(u;y) = \frac{1}{2 \sigma^2} \| \mb{A} u - y \|_2^2.$$
Then, Bayes' rule gives
$$
\frac{\dd \mu^y}{\dd \mu_0}(u) = \frac{1}{Z(y)} \exp\left( -\frac{1}{2\sigma^2} \|
  \mb{A}u - y\|_2^2
\right). 
$$
Now define the prior measure via
\begin{equation}\label{ellp-prior-baby-example}
\frac{\dd \mu_0}{\dd \Lambda}(u) = \frac{1}{U}\exp\left( -\| u\|_p^p\right) 
\end{equation}
where $\dd \Lambda$ denotes the Lebesgue measure on $\reals^n$, $\| \cdot
\|_p$ denotes the usual $\ell_p$ (quasi-)norm in  $\reals^n$ for $p > 0$ and $U$ is
the appropriate normalizing constant. Then the posterior $\mu^y$ can be 
identified via its Lebesgue density as
\begin{equation}\label{posterior-lebesgue-density}
\frac{\dd \mu^y}{\dd \Lambda}(u) = \frac{1}{Z(y)} \exp\left(   -\frac{1}{2\sigma^2} \|
  \mb{A}u - y\|_2^2 - \| u\|_p^p \right).
\end{equation}
The maximizer of the posterior density is referred to as the maximum
  a posteriori (MAP) estimate. Formally, the MAP estimate of the posterior in
  \eqref{posterior-lebesgue-density} is given by
$$
u_{\text{MAP}} = \argmin_{z \in \reals^n} \left\{ \frac{1}{2\sigma^2} \|
\mb{A}z - y
\|_2^2 + \| z\|_p^p \right\}.
$$
For $p \ge 1$ this optimization problem is convex and can be solved
efficiently. Taking $p=1$ results in the well-known
$\ell^1$-regularization technique which is often used in the recovery
of sparse solutions. For values of $p \in (0,1)$ the
resulting optimization problem is no longer convex but it is a good model for recovery of sparse or compressible solutions
\cite{foucart, lucka-dissertation}.
\end{example}

It is straightforward to check that the prior
distribution \eqref{ellp-prior-baby-example} for $p \in (0,1)$ is 
non-convex and heavy-tailed. However, we will see that this measure belongs to the much larger class of  infinitely divisible
measures. Formally, a random
variable $\xi$ is infinitely-divisible if for every $n\in \integers$
its law coincides with the law of $\sum_{k=1}^n \xi_k^{1/n}$ where
$\{\xi_{k}^{1/n}\}$ are i.i.d. random variables. 
Thus, the above example is our first attempt at demonstrating the
potential of infinitely-divisible prior measures (goals {\it G1} and {\it G2}) that are introduced in Section~\ref{sec:non-gaussian-priors}.
 The connection
between sparse recovery and heavy-tailed or infinitely divisible priors has been observed
in the literature. Unser and Tafti \cite{unser} and Unser et
al. \cite{unser-unified, unser-unified2} study the sparse behavior of
stochastic processes that are driven by infinitely divisible force
terms and advocate their use in solution of inverse problems. A
detailed discussion of some heavy-tailed prior distributions such as generalizations
of the student's-$t$ distribution and the $\ell_p$-priors can also be found in the
dissertation \cite{lucka-dissertation}. Finally, Polson and Scott \cite{scott-shrink} and
Carvalho et al. \cite{scott-horseshoe} propose a class of hierarchical
horseshoe priors 
that are tailored to the
recovery of sparse signals.

In practice, {\it solving a Bayesian inverse problem} often refers to
either identifying the posterior measure $\mu^y$ (such as in
\eqref{posterior-lebesgue-density}) or extracting certain statistics
from it such as the mean, the variance, maximizer of the density
etc. But before we can solve a Bayesian inverse problem we need to
know whether the problem is well-posed to begin with (point {\it G4} above):
Does
$\mu^y$ exist? Is it defined uniquely? Does it depend continuously on
the data $y$? And finally, can we approximate it in a consistent
manner?

 Later on we see that the well-posedness of a Bayesian inverse problem relies on the type of prior measure $\mu_0$
that is chosen during the modelling step as well as certain properties of the potential $\Phi$. 
Well-posed Bayesian inverse problems were studied in
\cite{stuart-acta-numerica, cotter-approximation} with Gaussian prior
measures, in \cite{dashti-besov} with Besov priors, in
\cite{hosseini-convex-prior} with convex prior measures and  in \cite{stuart-bayesian-lecture-notes, sullivan} 
with heavy-tailed priors on separable
Banach spaces.
\myhl{We note that our well-posedness results 
in this article are closely related to those of
\cite{stuart-bayesian-lecture-notes}. The main difference is that our theory 
does not rely on the assumption that the parameter space $X$ is
separable and we impose slightly different conditions on the potential $\Phi$}. The non-separability condition is particularly
interesting when
one takes $X$ to be $C^\alpha$(the space of H{\"o}lder continuous
functions) or $BV$, neither of which are separable.  In Section~\ref{sec:stochastic-process-prior} we introduce a class of prior measures that are concentrated on  
$BV$ and have piecewise constant samples (goal {\it G3}). This example is later used in Section~\ref{sec:examples} as 
a prior measure in a deblurring problem.

\subsection{Key definitions and notation} 
We gather here some key definitions and assumptions that are used in the remainder of the article. We let $\reals_+$ denote the positive real line $[0, \infty)$ and
use the shorthand notation $a \lesssim b$ when $a$ and $b$
are real valued functions
 and there exists an independent constant $C > 0$ such that $ a \le C
 b$. Given two random variables $\xi$ and $\zeta$ 
we use the notation $\xi \dequal \zeta$ to denote that they have the same laws (or
distributions). 

We use the shorthand notation $\{ \gamma_k \}$
to denote a sequence of elements $\{ \gamma_k \}_{k=1}^\infty$ in a
vector space. 
The usual $\ell_p$ sequence spaces  for $p \in [1, \infty]$
are defined as the space of real valued sequences $\{ \gamma_k\}$  such that  $\|
\{ \gamma_k\} \|_p < \infty$ where
$$
\left.
\begin{aligned}
\|
\{ \gamma_k \} \|_p := \left( \sum_{k=1}^\infty |\gamma_k|^p
\right)^{1/p} \qquad\text{if} \qquad p \in [1, \infty)   \qquad
\text{and} \qquad
\| \{ \gamma_k \} \|_\infty := \sup_{k} |\gamma_k|.
\end{aligned}  
\right.
$$
Similarly, we define the $\| \cdot \|_p$ norms of
finite dimensional vectors. In particular $\| \cdot \|_2$ will denote the usual
Euclidean norm. Given a positive definite matrix 
$\pmb{\Sigma}$ of size $m
\times m$,  we define the norm 
$$
\| x \|_{\pmb{\Sigma}} := \| \Sigma^{-1/2} x \|_2 \qquad \text{for}
\qquad x \in \reals^m.
$$

Throughout the article we use $\Lambda$ to denote the Lebesgue
measure in finite dimensions.
Given a Borel
 measure $\mu$ on a Banach space $X$ we define the spaces $L^p( X, \mu)$ for $p \in
 [1, \infty)$ 
as the space of $\mu$-equivalent classes of functions $h: X \mapsto
\reals$ such that $|h|^p$ is $\mu$-integrable. We also use the shorthand
notation $L^p(X)$ instead of $L^p(X, \Lambda)$ whenever we are working
with the Lebesgue measure. 
 Finally, if $X$ is a Banach
space, we use $X^\ast$ to denote the topological dual of $X$ and \myhl{$B_X(r)$ to denote the open ball of 
radius $r>0$ in $X$ that is centered at the origin. The shorthand notation $B_X$ denotes 
the unit ball.}

We shall consider the prior probability measure $\mu_0$ to be in the class of
 \myhl{Borel} probability measures on $X$. \myhl{In some cases we assume that the prior $\mu_0$ is Radon
 meaning that it is an inner
 regular probability measure on the Borel sets of $X$.} Furthermore,
 whenever we say that $\mu$ is
a probability  measure on $X$ we automatically mean that $\mu(X)
=1$. Finally, throughout this article we only consider complete probability measures in the following sense: If $\mu$ is a
Borel probability measure on $X$ and $A$ is a set of $\mu$-measure
zero then every subset of $A$ also has measure zero.

In this article we focus on the following notion of
a  well-posed Bayesian inverse problem:

\begin{definition}[Well-posedness] \label{def-existence-uniqueness}
Suppose that $X$ is a Banach space and $d( \cdot, \cdot) \mapsto \reals$
is a metric on the space of
Borel probability measures on $X$.
Then for a choice of the
prior measure $\mu_0$ and the likelihood potential $\Phi$, the Bayesian inverse problem given by \eqref{bayes-rule}
 is well-posed with respect to $d$ if:
  \begin{enumerate}
  \item (Existence and uniqueness) There exists a unique posterior
    probability measure $\mu^y \ll \mu_0$ given by Bayes' rule \eqref{bayes-rule}.
  \item (Stability) For every choice of $\epsilon >0$ there exists a
    $\delta > 0$ so that 
 $d( \mu^y , \mu^{y'} ) \le \epsilon$ for all $y, y' \in Y$ so that
 $\| y - y' \|_Y \le \delta$.
  \end{enumerate}
\end{definition}

We will study the
convergence of probability measures using
 the Hellinger and total variation
metrics
on the space of probability measures on $X$. 
For two probability measures $\mu_1$ and $\mu_2$ that are absolutely
continuous with respect to  
a third measure $\nu$ on $X$, the total variation and Hellinger metrics are defined as 
\begin{equation}
  \label{TV-metric}
  d_{TV}( \mu_1, \mu_2) := \frac{1}{2} \int_X   \left|  \frac{\dd
      \mu_1}{ \dd \nu}  - \frac{\dd \mu_2}{ \dd \nu}  \right| \dd \nu
  \quad \text{and} \quad   d_H( \mu_1, \mu_2) := \left( \frac{1}{2} \int_X  \left( \sqrt{ \frac{\dd \mu_1}{ \dd \nu} } - \sqrt{ \frac{\dd \mu_2}{ \dd \nu} }  \right)^2 \dd \nu \right)^{1/2}.
\end{equation}
Both metrics are independent of the choice of the measure $\nu$
 \cite[Lem.~4.7.35]{bogachev1}. 
Furthermore, convergence in one of these metrics implies convergence
in the other, due
to the following inequalities (see 
\cite[Lem.~4.7.37]{bogachev1} for a proof) 
\begin{equation}\label{tv-hellinger-equivalence}
2 d^2_H(\mu_1 , \mu_2) \le d_{TV} (\mu_1, \mu_2) \le \sqrt{8}
d_H(\mu_1, \mu_2).
\end{equation}
However, one might prefer to work with the Hellinger metric as it relates
directly to the error in expectation of certain
functions. Suppose that $h \in L^2(X, \mu_1) \cap L^2(X, \mu_2)$. Then
using the Radon-Nikodym theorem and H\"{o}lder's inequality one can
show (see \cite[Sec.~1]{hosseini-convex-prior} for details)
\begin{equation}\label{hellinger-metric-bounds-expectation}
\begin{aligned}
 \left| \int_X h(u)  \dd \mu_1 (u)  \right. \left.- \int_X h(u) \dd \mu_2(u) \right|
\le 2\left( \int_X h^2(u) \dd \mu_1 + \int_X h^2(u) \dd \mu_2\right)^{1/2}
d_H(\mu_1, \mu_2). 
\end{aligned}
\end{equation}
For reasons that will become clear in Section
\ref{sec:well-posedness}, we prefer to study the well-posedness of
inverse problems using both the Hellinger and total variation
metrics. The main difference is in the restrictions that we need to
impose on the prior $\mu_0$ in order to obtain a certain rate of convergence for each
metric. 


\section{Infinitely-divisible prior measures} 
\label{sec:non-gaussian-priors}

 We start by presenting a generalization of the prior
distribution \eqref{ellp-prior-baby-example} that was considered in
Example 1. We show that this prior belongs to a larger class of
distributions that are closely related to $\ell_p$ 
regularization techniques. We shall extend these distributions to
measures on Banach spaces with an unconditional Schauder
basis and observe that they belong to
the much larger class of infinitely-divisible (ID) measures (see Definition \ref{definition-ID-measures}).
 Motivated
by this connection between $\ell_p$ regularization and ID priors, we
turn our attention to the ID class and discuss some of its properties. In particular, we study the tail behavior of ID priors 
with respect to their L{\'e}vy measures (see Definition \ref{levy-measure-definition}).




\subsection{A class of shrinkage priors with compressible samples}\label{sec:priors-for-compressibility}
When faced with the problem of recovering a sparse or compressible
parameter we require the prior measure to
reflect the intuition that the solution to the inverse problem is likely to
have only a few large modes in some basis and the rest of the modes are
negligible (see
\cite[Sec.~6.1]{lucka-dissertation}). 
Such prior distributions are often referred to as ``shrinkage
priors'' and they have been the subject of extensive research
\cite{scott-shrink, scott-horseshoe, ghosh,
  vandervaart-needles-in-haystack, vandervaart-bayesian-sparse}. In
this section we consider a few examples of shrinkage priors that are
closely related to $\ell_p$ regularization techniques.

Most of the existing literature on  shrinkage priors is
focused on finite dimensional problems but we present an extension
of these priors to infinite-dimensional 
Banach spaces.
Since compressibility is often considered with
respect to a basis, it makes sense for us to consider 
a parameter space that has a basis. 


Given a parameter space $X$, or at least a subspace $\tilde{X}
\subseteq X$ that has an unconditional Schauder basis $\{ x_k \}$, we
construct random variables of the form 
\begin{equation}
\label{product-prior-sample}
u \sim \sum_{k=1}^\infty \gamma_k \xi_k x_k
\end{equation}
where $\{ \gamma_k \}$ is a fixed sequence of real valued coefficients
that decay sufficiently fast and the $\{\xi_k\}$ are a sequence of
independent real valued random variables that need not be
identically distributed. We will take
the prior measure $\mu_0$ to be the law of the random variable $u$ in \eqref{product-prior-sample}. We 
refer to such a prior measure $\mu_0$ as the {\it product prior}
obtained from $\{ \gamma_k\}$ and $\{ \xi_k\}$. This
construction of the prior is reminiscent of the Karhunen-Lo\'{e}ve
expansion of Gaussian measures \cite[Thm.~3.5.1]{bogachev-gaussian}. The following theorem
gives sufficient conditions that ensure $\| \cdot \|_X < \infty$ $\mu_0$-a.s. 

\begin{theorem}\cite[Thm.~3.9]{hosseini-convex-prior}
\label{product-prior-ellp}
  Suppose that $X$ is a Banach space with an unconditional Schauder
  basis and let $u$
  be  as in \eqref{product-prior-sample}. If $\{ \gamma_k^2 \}
  \in \ell_p$ and $\{ \VV \xi_k \} \in \ell^q$ for $1 < p, q <
  \infty$ so that $1/p + 1/q = 1$ (with $p=1$ for the limiting case
  when $q=\infty$), then $\| u \|_X < \infty $ a.s. In particular, 
if the $\{ \xi_k\}$ are i.i.d.,  $\VV \xi_1 < \infty$ and $\{ \gamma_k\} \in \ell^2$, then 
$\| u\|_X < \infty$ a.s.
\end{theorem}

\myhl{We can also show that the product prior $\mu_0$ that is induced by
\eqref{product-prior-sample} is Radon. Proof of the next theorem 
follows the same approach as \cite[Thm.~3.10(ii)]{hosseini-convex-prior} and is hence omitted. }

\begin{theorem}\label{product-prior-is-radon}
  Let $\mu$ be the probability measure that is induced by the random
  variable $u$ given by \eqref{product-prior-sample} where $\{
  \gamma_k \} $ and $\{ \xi_k\}$ satisfy the conditions of Theorem
  \ref{product-prior-ellp}. Then $\mu$ is a Radon probability measure
  on $X$ if the random variables $\{\xi_k\}$ are distributed according
  to Radon
  probability measures on $\reals$. 
\end{theorem}

Before going further we present a result on the second raw
moment of product priors which will be useful throughout the remainder
of the article.
\begin{theorem}\label{bounded-moments-product-prior}
  Suppose that $X$ is a Banach space with an unconditional Schauder
  basis $\{ x_k\}$ and
  let $\mu$ be the product prior obtained from $\{ \gamma_k\} \in 
\ell^2 $ and
  $\{ \xi_k \}$ where $\xi_k$ are i.i.d. and $\VV \xi_k < \infty$. Then
  $\|\cdot\|_X \in L^2 (X, \mu)$.
\end{theorem}
\begin{proof}
Let $u_N = \sum_{k=1}^N \gamma_k \xi_k x_k$ then for  $M > N > 0$ we
have 
$$
\begin{aligned}
\left| \int_X  \| u_M\|_X^2  ~\dd \mu - \int_X \| u_N\|_X^2~ \dd \mu \right| &= 
\left|\int_X (\| u_M\|_X - \| u_N\|_X) (\| u_M\|_X + \| u_N\|_X) ~ \dd
\mu\right|
\end{aligned}
$$
By Theorem~\ref{product-prior-ellp} we know that $\| u \|_X <
\infty$ a.s. and so in the limit as $M,N \mapsto \infty$, $|(\| u_M\|_X -
\| u_N\|_X)| \mapsto 2 \| u\|_X$ and  $|(\| u_M\|_X -
\| u_N\|_X)| \mapsto 0$ and so $\{ \| u_N\|^2_X \}$ is
Cauchy in $L^2(X, \mu)$. 
\end{proof}

We are now in position to discuss a few examples of shrinkage
priors. 
Motivated by Example \ref{example-1}, we define the class of $\ell_p$-priors
 as follows:
\begin{definition}[$\ell_p$-prior]
 ~Let $X$ be a Banach space with an unconditional Schauder basis $\{
 x_k \}$, then we say that a Radon probability measure
$\mu$ is an $\ell_p$-prior on $X$ if its samples can be expressed as
$
u = \sum_{k=1}^\infty \gamma_k \xi_k x_k
$
where $\{\gamma_k\} \in \ell^2$ and $\{ \xi_k \} $ is an i.i.d.
sequence of real valued random variables with Lebesgue density
\begin{equation}\label{generalized-normal}
\xi_k \sim \frac{p}{2 \alpha \Gamma(1/p)}\exp\left(
-{\frac{|t|^p}{\alpha^p} }\right) \dd \Lambda(t)
\end{equation}
where $p \in (0, \infty)$ and $\alpha = \sqrt{ \Gamma(1/p) / \Gamma(3/p) }$.
\end{definition}

Here $\Gamma$ denotes the usual Gamma function.
The distribution in \eqref{generalized-normal} belongs to the larger
class of Generalized Normal distributions
\cite{nadarajah-generalized-normal}.
This class is also referred to as a Kotz-type
distribution \cite{nadarajah-kotz} or a generalized Laplace
distribution \cite{kotz-laplace}. Here we will not use either of these
terms and simply refer to this distribution as the
$\ell_p$-distribution to emphasize its connection to
$\ell_p$-regularization techniques. The random variables $\xi_k$ have bounded moments
of all orders (see
\cite{nadarajah-generalized-normal} or the discussions following the
definition of the $G_{p,q}$-prior below), in fact
$$
\EE \xi_k^s = \frac{\alpha^s ( 1 + (-1)^s)}{2\Gamma(1/p)}
\Gamma\left( \frac{s + 1}{p} \right) \qquad \text{for} \qquad s \in \integers.
$$
In particular we have that $\VV \xi_k =
1$ and so it follows from Theorem \ref{bounded-moments-product-prior}
that the $\ell_p$-prior has bounded second moments.


Another, closely related class of priors to the $\ell_p$-priors can be obtained by a symmetrization of the
Weibull distribution:
\begin{definition}[$W_p$-prior]
  Let $X$ be a Banach space with an unconditional Schauder basis $\{
  x_k \}$, then we say that a Radon probability measure
$\mu$ is a $W_p$-prior on $X$ if its samples can be expressed as
$
u = \sum_{k=1}^\infty \gamma_k \xi_k x_k
$ where $\{\gamma_k\} \in \ell^2$ and $\{ \xi_k \} $ is an i.i.d.
sequence of real valued random variables with Lebesgue density 
\begin{equation}\label{Wp-distribution}
\xi_k \sim \frac{p}{\alpha} \left( \frac{|t|}{\alpha} \right)^{p-1} 
\exp \left( -\frac{|t|^p}{\alpha^p} \right) \dd \Lambda(t),
\end{equation}
where $p \in (0, \infty)$ and $\alpha = (2\Gamma(1 + 2/p))^{-1/2}$.
\end{definition}

The distribution of $\xi_k$ is simply a symmetric version of
the well-known Weibull distribution
\cite{kotz-univariate-v1}, hence the name $W_p$. A straightforward calculation shows that
$\VV \xi_k = 1$ and once again it follows from Theorem
\ref{bounded-moments-product-prior}
that the $W_p$-priors have bounded second moments.

Both the $W_p$ and $\ell_p$
distributions reduce to the Laplace distribution when $p=1$. For $p <
1$ the $\ell_p$ distribution has non-convex level sets and puts a
large portion of
its mass close to the axes (see Figure~\ref{fig:ellp-wp-samples}). This behavior becomes stronger for smaller  $p$
 and suggests that the $\ell_p$-prior will incorporate
sparse behavior as $p \mapsto 0$. 

The $W_p$ distribution behaves very differently in comparison
to the $\ell_p$ distribution.  For
$p< 1$ the $W_p$ distribution blows up at the origin (see Figure~\ref{fig:ellp-wp-samples}(a)). This means that the $W_p$ distribution
puts more of its mass at the
origin which leads us to believe that it must incorporate stronger
compressibility than the $\ell_p$ distribution. 

Further insight into the behavior of the $W_p$-prior can be obtained
by considering its MAP point estimate in finite dimensions. Formally, using this prior in
Example \ref{example-1}
gives rise to an optimization problem of the
form
$$
u_{\text{MAP}} = \argmin_{z \in \reals^n}  \left\{ \frac{1}{2} \| \mb{A} z -
y \|_2^2 + \| z\|_p^p + (1 - p) \sum_{k=1}^n \log( | z_k|) \right\}. 
$$
Of course, the $\log$ term on right hand side is not bounded from
below and so we cannot gain much insight from this problem.
 However, we can consider a slightly modified version of this optimization
problem by introducing a small parameter $\epsilon >0$ 
$$
u_\epsilon = \argmin_{z \in \reals^n}  \left\{ \frac{1}{2} \| \mb{A} z -
y \|_2^2 + \| z\|_p^p + (1 - p) \sum_{k=1}^n \log( \epsilon + | z_k|) \right\}.
$$
Now if $\epsilon$ is small then the $\log$ term will heavily
penalize any modes of the solution that  are on a larger scale than that of
$\epsilon$ and so we expect that most of the modes of the solution
$u_\epsilon$ will be on the scale of the small parameter
$\epsilon$. The stronger shrinkage of the posterior due to the $W_p$-prior is also evident in
Figure \ref{fig:generic-densities-2D} where we compare a prototypical example of  posteriors
that arise from the $W_{p}$ and $\ell_p$ priors for solution of
Example \ref{example-1} in 2D. Here, we clearly see that the
$W_{1/2}$-prior results in a posterior that is highly concentrated
around the axes compared to the posterior that arises from
$\ell_{1/2}$-prior 
which is more spread out. Note that in either case, the posteriors are
highly concentrated around the axes meaning that the map estimates as 
well as most of the samples from these posteriors will incorporate sparsity.

\begin{figure}[htp]
  \centering
        \includegraphics[width=.3\textwidth]{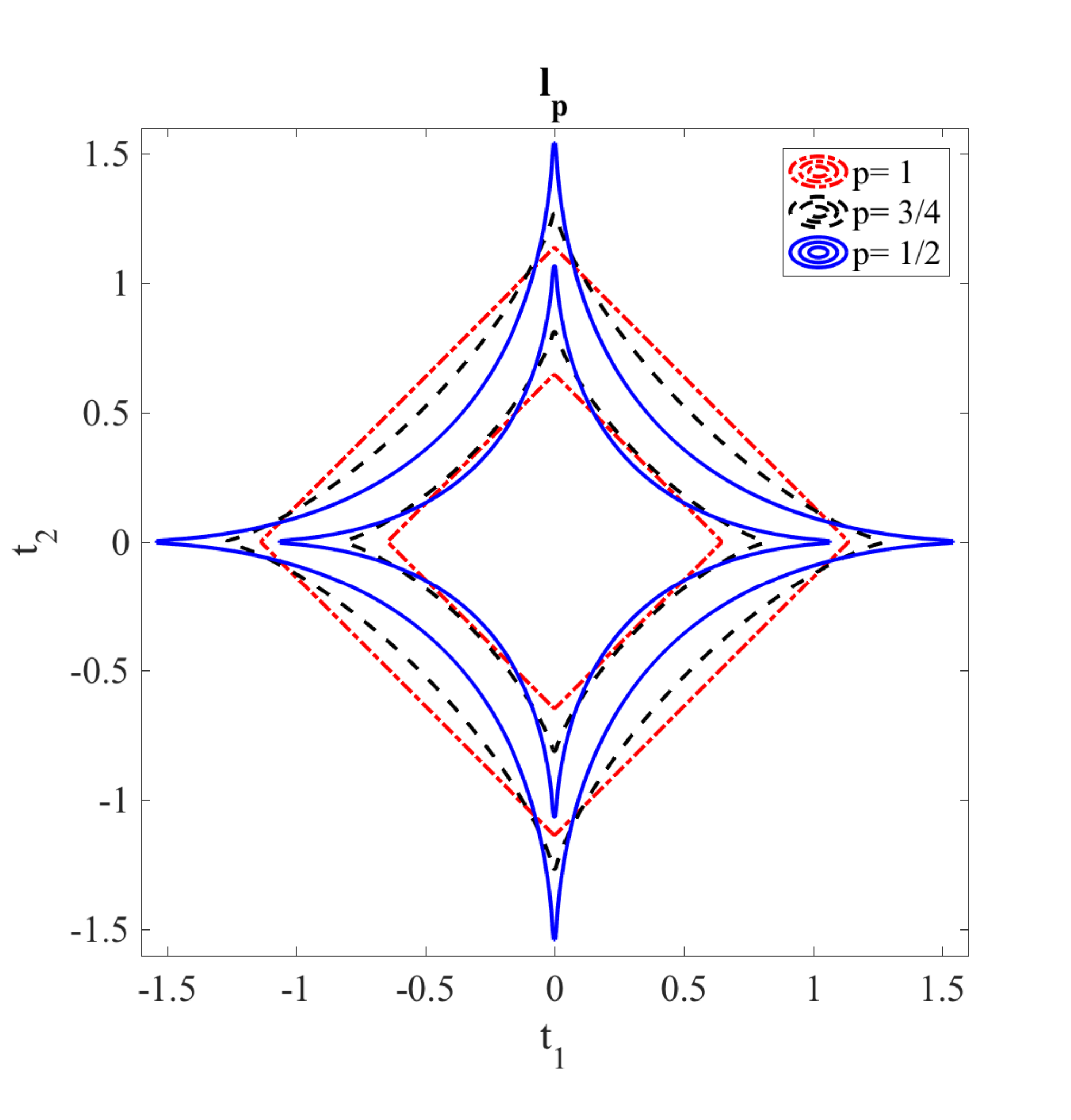}\qquad
        \includegraphics[width=.3\textwidth]{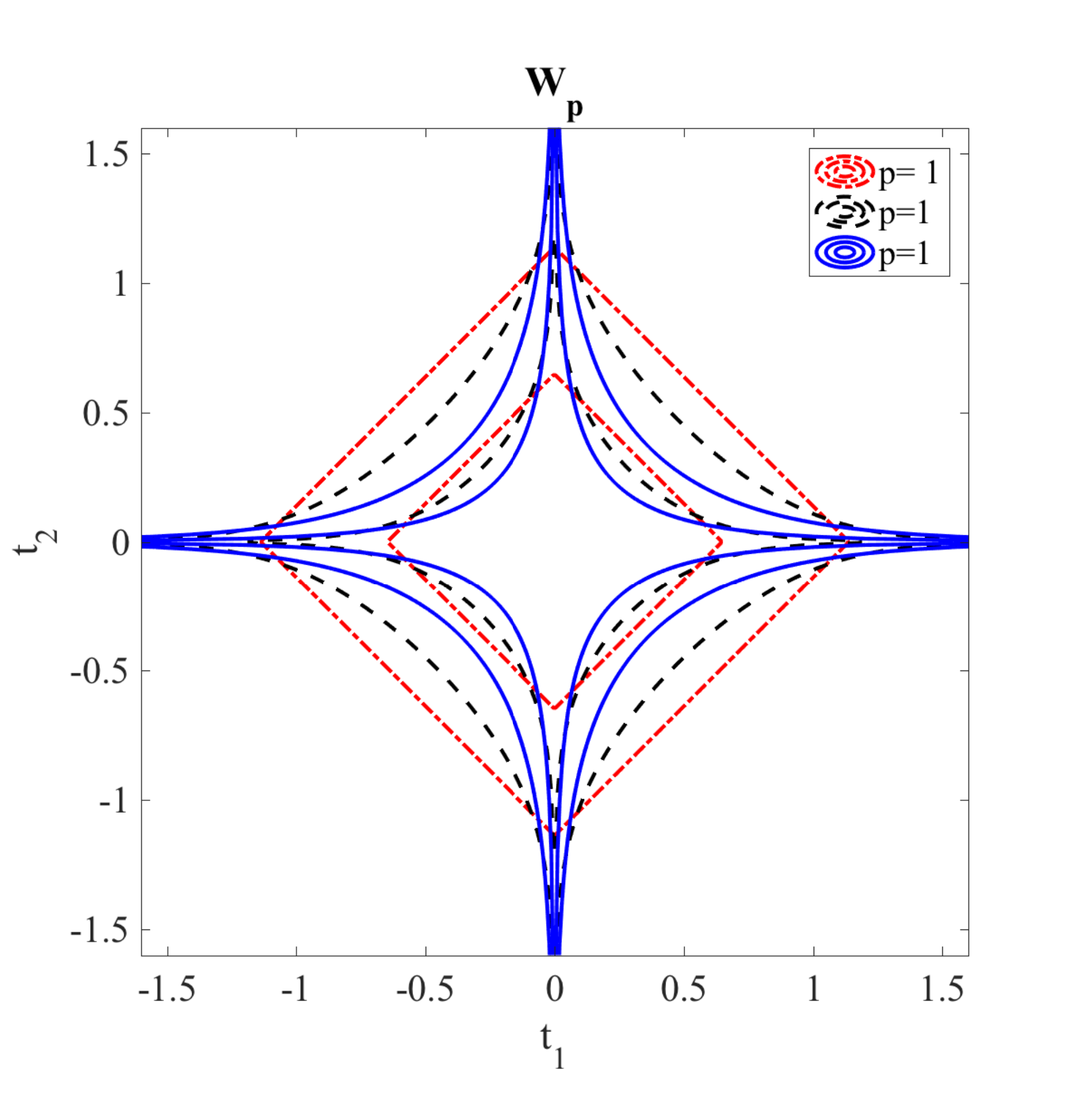}
  \caption{Contour plots of $\ell_p$ and $W_p$
    densities in 2D for different values of $p$.}
  \label{fig:ellp-wp-samples}
\end{figure}

\begin{figure}[htp]
  \centering
  \includegraphics[width=0.32\textwidth, clip=true, trim= 2cm 0cm 3cm 0cm]{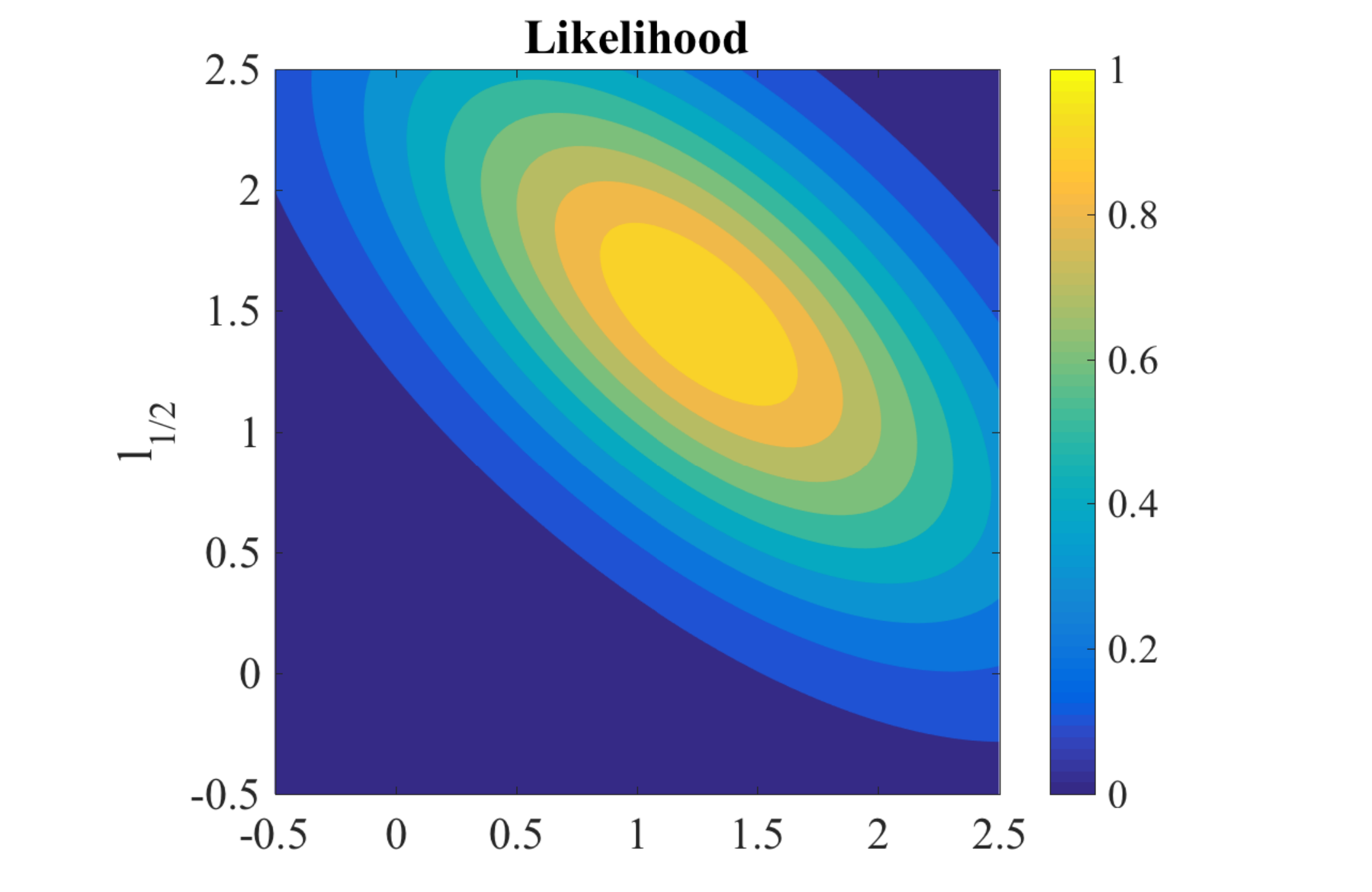}
  \includegraphics[width=0.32\textwidth, clip=true, trim= 2cm 0cm 3cm 0cm]{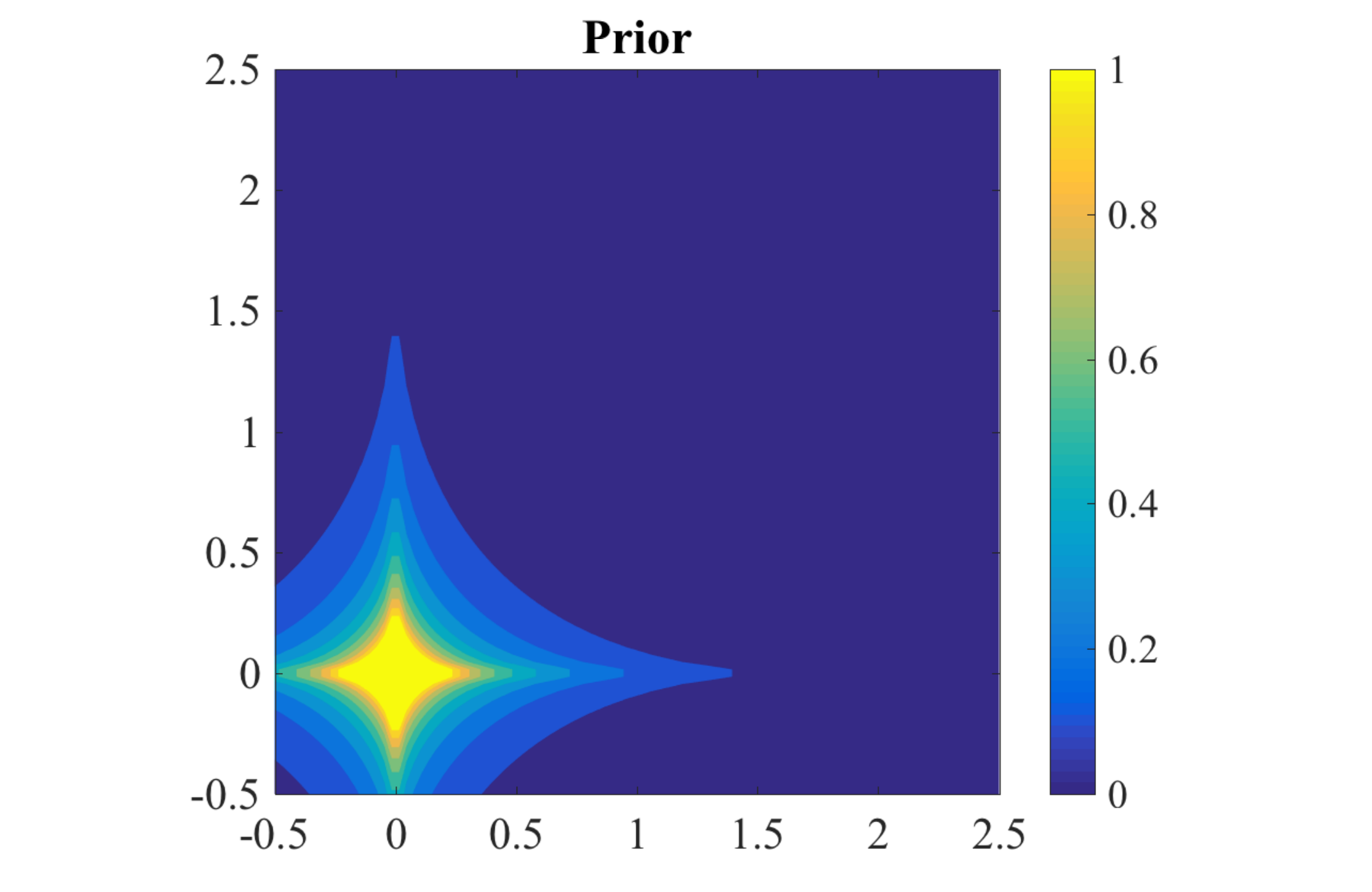}
  \includegraphics[width=0.32\textwidth, clip=true, trim= 2cm 0cm 3cm
  0cm]{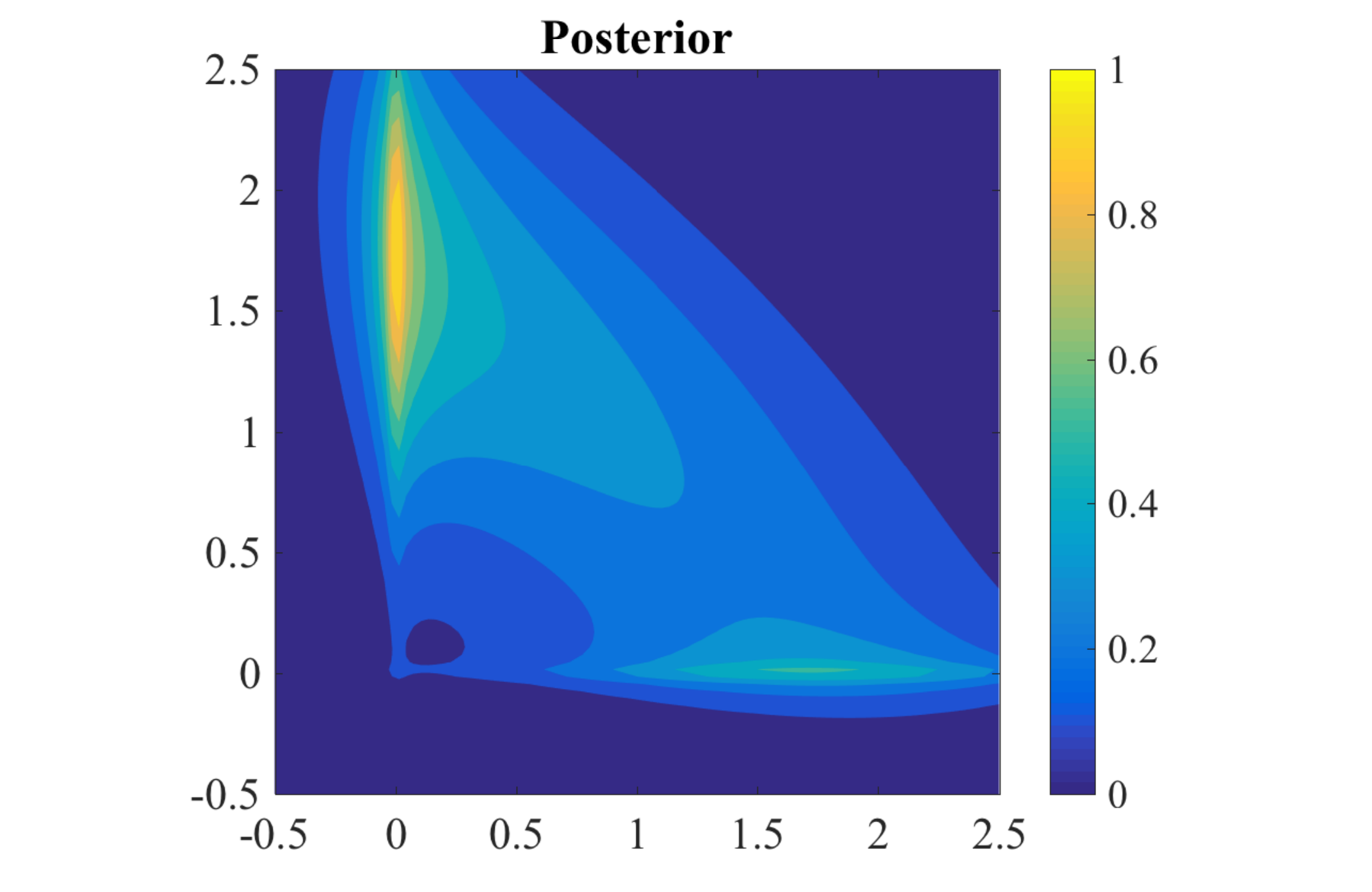} \\
  \includegraphics[width=0.32\textwidth, clip=true, trim= 2cm 0cm 3cm 0cm]{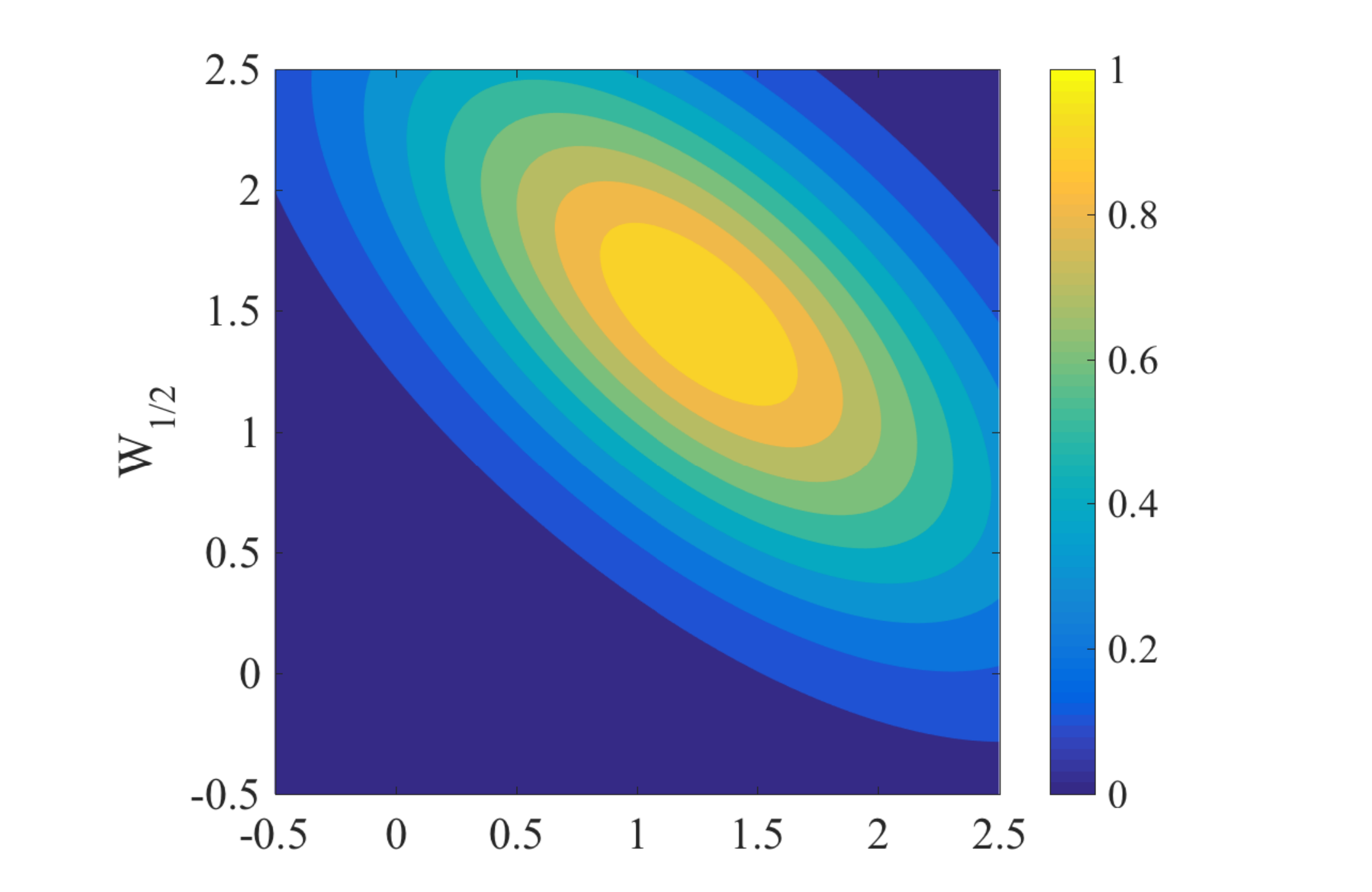}
  \includegraphics[width=0.32\textwidth, clip=true, trim= 2cm 0cm 3cm 0cm]{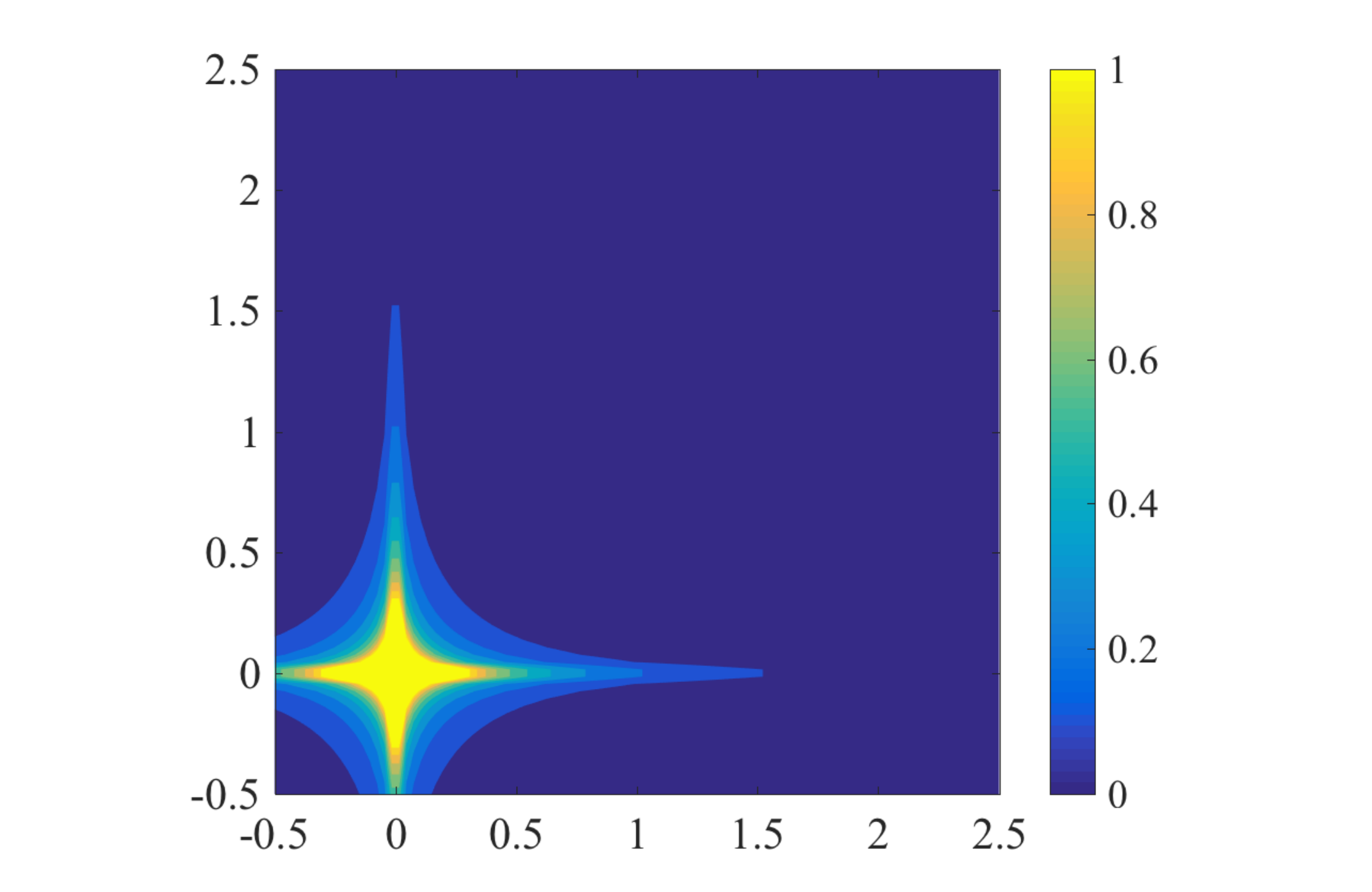}
  \includegraphics[width=0.32\textwidth, clip=true, trim= 2cm 0cm 3cm
  0cm]{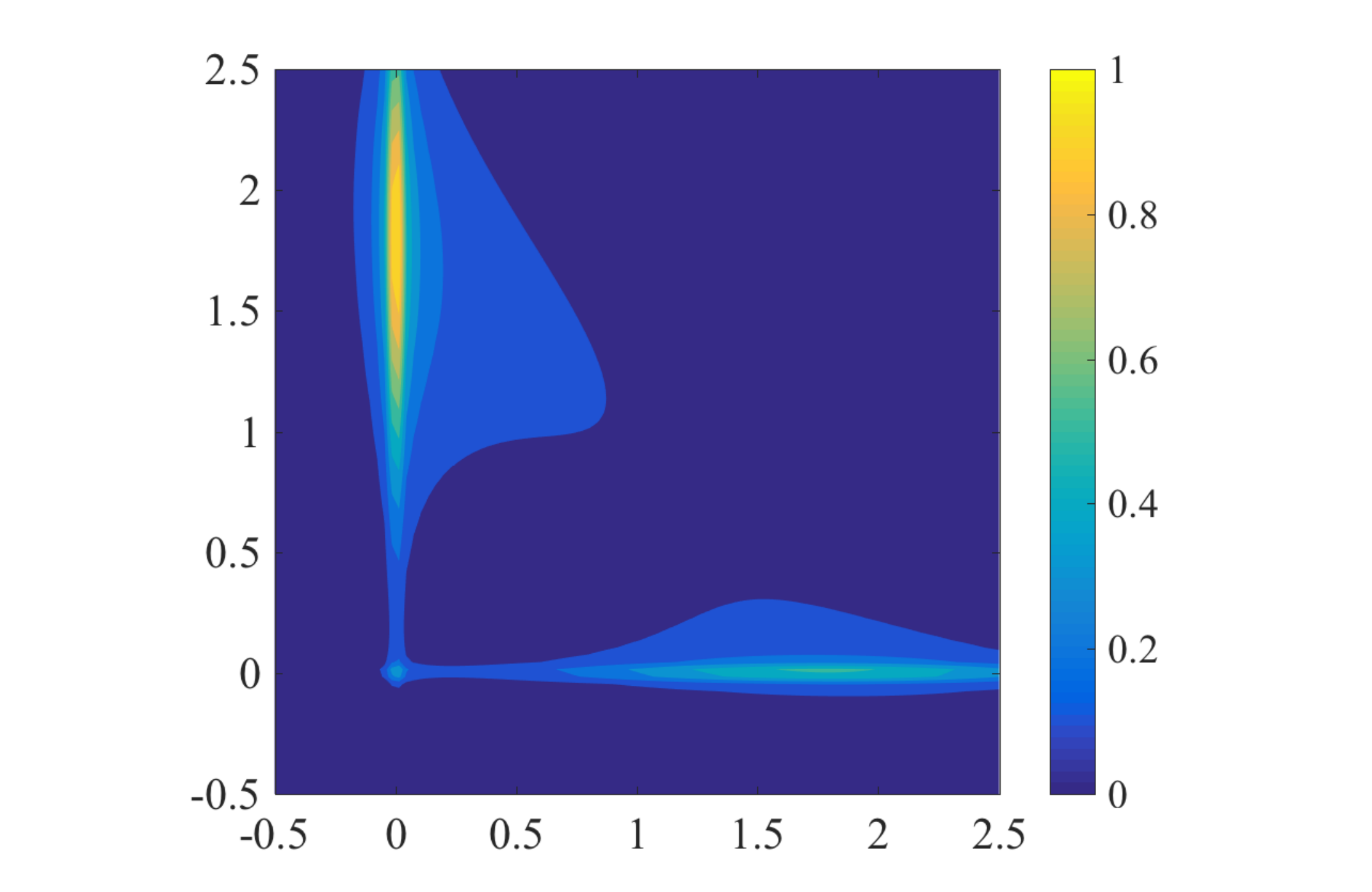} 
  \caption{A prototypical example of densities that arise in the
    solution of Example \ref{example-1} in 2D with the $\ell_{1/2}$ (top row) and
    $W_{1/2}$ priors (bottom row). From left to right columns: The
    likelihood that arises from the additive Gaussian noise model, the
    prior densities and the resulting posteriors. The densities
    are rescaled for
    better visualization.
}
  \label{fig:generic-densities-2D}
\end{figure}

Comparing the distributions \eqref{generalized-normal} and
\eqref{Wp-distribution} suggests the definition of a larger class of
priors that can interpolate between the $\ell_p$ and $W_p$-priors. 
To this end, we introduce a new class of prior measures called the
$G_{p,q}$-priors. The letter
$G$ is chosen due to the connection of the one dimensional version of
these measures to the generalized Gamma distribution \cite{bondesson}.

\begin{definition}[$G_{p,q}$-prior]
  Let $X$ be a Banach space with an unconditional Schauder basis $\{x_k\}$, then we say that a Radon probability measure
$\mu$ is a $G_{p,q}$-prior on $X$ if its samples can be expressed as
$
u = \sum_{k=1}^\infty \gamma_k \xi_k x_k
$
with $\{\gamma_k\} \in \ell^2$ and $\{ \xi_k \} $ is an i.i.d.
sequence of real valued random variables with Lebesgue density
\begin{equation}\label{Cpq-distribution}
\xi_1 \sim \frac{p}{2\alpha \Gamma(q/p)} \left| \frac{t}{\alpha} \right|^{q-1} \exp
\left( - \left|\frac{t}{\alpha} \right|^p \right)d\Lambda(t),
\end{equation}
where $p \in (0, \infty)$ and $\alpha = (\Gamma(q/p)/\Gamma((2 + q)/p))^{1/2}$.
\end{definition}

Using the change of variables $s =
\frac{t^p}{\beta^p}$ we see that for $k, \beta \ge 0$
$$
\begin{aligned}
\int_{0}^\infty t^k \left(\frac{t}{\beta}\right)^{q-1} \exp\left(
  -\left(  \frac{t^p}{\beta^p}\right)
\right) d\Lambda(t)  =\frac{\beta}{p}  \int_{0}^\infty s^{\frac{k + q }{p} -1} \exp\left(
  -s
\right) d\Lambda(s) = \frac{\beta^{k+1}}{p}  \Gamma\left( \frac{k + q}{p} \right).
\end{aligned}
$$
Setting $k =0$ leads us to the normalizing constant in the definition
of the distribution in \eqref{Cpq-distribution}. Furthermore, we
obtain the following expression for the moments of the $G_{p,q}$ distributions
$$
\EE |\xi_1|^s =  \frac{\alpha^s( 1 + (-1)^s) \Gamma\left( (s +
      q)/p \right) }{2\Gamma\left(q/p\right)} \qquad s \in \integers.
$$
In particular 
$
\VV \xi_1 = {\alpha^2 \Gamma( (2 + q)/p)/\Gamma(q/p)} = 1. 
$
Clearly, the $\ell_p$ prior is equivalent to $G_{p,1}$ and 
$W_p$ is equivalent to $G_{p,p}$. Furthermore, the $G_{1,q}$
distribution coincides with a symmetrization of the the Gamma distribution.
 For $q < 1$
the distribution \eqref{Cpq-distribution} will blow up at the origin and so it
will put a lot of its mass at zero. 
The $G_{p,q}$ distributions belong to the
class of ID measures by the following theorem of Bondesson.

\begin{theorem}[{\cite[Cor.~2]{bondesson}}]
  ~All probability density functions on $(0, \infty)$ of the form 
$$
\pi(t) = \frac{p}{\alpha \Gamma(q/p)} \left( \frac{t}{\alpha} \right)^{q-1} \exp
\left( - \left(\frac{t}{\alpha} \right)^p
 \right) 
$$
are ID for $q, \alpha > 0$ and $0<p \le 1$.
\end{theorem}

Later on we show that if the $\xi_k$ are distributed according to an
ID distribution then
the corresponding product prior on $X$ will also be an ID probability measure. Then the $G_{p,q}$-priors are also ID.
This fact suggests the question of what other types of
ID measures are good models for compressibility? We know that heavy-tailed
distributions such as the Cauchy or Student's t distributions are 
ID and they incorporate compressible samples as well. Then there is
much to be gained from the study of ID prior measures in Bayesian
inverse problems.
To the best of our knowledge a thorough study of the compressible
behavior of ID distributions is still missing in the
literature. The closest reference in this direction is
 the works of Unser et. al. \cite{unser, unser-unified,
   unser-unified2}. While we do not study the modelling of
 compressible parameters, we recognize the potential impact of ID priors in
 this subject and so we dedicate the remainder of this section to
the study of ID priors.



\subsection{Infinitely divisible priors}
We begin by collecting some results on the class of
ID probability measures on Banach spaces. 
We only present the results that are needed in our exposition and refer
the reader to \cite{linde} for
 a detailed introduction to ID measures on Banach spaces. Further
 reading can be found in the monograph \cite{unser} which contains a
modern treatment of ID probability measures on
 nuclear spaces and
 the books \cite{applebaum, sato,steutel} that
 are good references on the theory of ID measures in finite
 dimensions. 

Recall that
given a Borel probability measure $\mu$ on a Banach space $X$ 
its characteristic function $\hat{\mu} : X^\ast \mapsto \complex$ is given by
$$
\hat{\mu}(\varrho) = \int_X \exp( i \varrho(u) ) \dd \mu(u) \qquad \forall
\varrho \in X^\ast.
$$
Characteristic functions play a crucial role in our
discussion of ID measures in this section. In what
follows $\nu^{\ast n}$ denotes the $n$-fold convolution
of a measure $\nu$ with itself.



 \begin{definition}[ID measures {{\cite{linde}}}] \label{definition-ID-measures}
A Radon probability measure $\mu$ on a Banach space $X$ is 
called an infinitely divisible measure if for each $n \in \integers$ there exists a
Radon probability measure $\mu_{1/n}$ so that 
$
\mu = (\mu_{1/n})^{\ast n}.
$
Equivalently, the probability measure $\mu$ is ID
if $ \hat{\mu}(\varrho) = (\hat{\mu}_{1/n}(\varrho))^n, \forall \varrho \in X^\ast.
$
 \end{definition}

Put simply, a real valued random variable $\xi$ is distributed
according to an ID measure if for every $n \in \integers$ one can
find a collection of i.i.d. random variables $\{\xi_k\}_{k=1}^n$ so that
 $\xi \dequal \sum_{k=1}^n \xi_{k}$. Examples of such distributions
 include Gaussian, Laplace, Gamma,  log-normal, Cauchy and
 Student's-t.
More examples can be found in the monograph
\cite{steutel} where ID distributions on $\reals$ are studied in
detail. We note that an equivalent definition of an ID measure
is given
as
the
 law of a L{\'e}vy process terminated at unit time. 
However,  we will
 not use this definition in order to avoid the technicalities of
 dealing with
 L{\'e}vy processes
but instead we refer the interested reader to
 the monographs \cite{peszat, tankov} for further reading.
The proof of the next theorem can be found in \cite[Sec.~5.1]{linde}. 

\begin{theorem}\label{id-measure-properties}
Let $\mu$ be an ID probability measure on a Banach space $X$. Then 
\begin{enumerate}[(i)]
\item $\hat{\mu}(\varrho) \neq 0$ for all $\varrho \in X^\ast$.
\item There exists a unique and continuous (in the dual norm) function $\psi: X^\ast \mapsto
  X$ so that $\hat{\mu}(\varrho) = \exp( \psi(\varrho))$ and $\psi(0)= 0$.
\item If $\mu$ is symmetric, i.e. $\mu(A) = \mu(-A)$ for all Borel
  subsets $A$ of $X$,
then $\hat{\mu}$ is real valued and
  positive.
\item  For every $n \in \integers$ the measures $\mu_{1/n}$ are
  uniquely determined and 
$\hat{\mu}_{1/n}(\varrho) = \exp( n^{-1} \psi(\varrho))$ for all $\varrho \in X^\ast$.
\end{enumerate}
\end{theorem}
Furthermore, we define the function  
$$
\Psi(u, \varrho) := \exp( i
\varrho (u)) - 1 - i \varrho (u) \mb{1}_{B_X} (u)  \qquad \forall u
\in X, \varrho \in X^\ast,
$$
where $B_X$ is the unit ball in $X$ and $\mb{1}_{B_X}$ is the
characteristic function of the unit ball. We recall the
definition of a
 L{\'e}vy measure on a Banach space.
\begin{definition}[L{\'e}vy Measure] \label{levy-measure-definition}
  A positive  $\sigma$-finite Radon measure $\lambda$ on $X$ is called a L{\'e}vy measure
  if and only if 
  \begin{enumerate}
  \item $\lambda(\{ 0 \}) = 0$.
\item $\int_X | \Psi( u, \varrho) | \dd \lambda(u) < \infty$ for every
  $\varrho \in X^\ast$. 
\item $\exp( \int_X \Psi(u, \varrho) \dd \lambda(u))$ is the
  characteristic function of a Radon probability measure on $\reals$ for every
  $\varrho \in X^\ast$.
  \end{enumerate}
\end{definition}
We are now ready to present the celebrated L{\'e}vy-Khintchine representation
theorem (see \cite[Sec.~5.7]{linde} for a proof):
\begin{theorem}[L{\'e}vy-Khintchine representation]\label{levy-khintchine-representation}
  A Radon probability measure on a Banach space $X$ is infinitely
  divisible if and only if there exists an element $m \in X$, a
  (positive definite) covariance operator $\mcl{R}: X^\ast \mapsto X$ and a
 L{\'e}vy measure $\lambda$, so that 
\begin{equation}\label{levy-khintchine-formula}
\hat{\mu}(\varrho) = \exp( \psi(\varrho)) \quad \text{where} \quad
\psi(\varrho) = i \varrho (m) - \frac{1}{2} \varrho ( \mcl{R}(\varrho)) + \int_X \Psi( u,
\varrho) \dd \lambda(u).
\end{equation}
Equivalently, $\mu$ is ID precisely when there
exists a point mass $\delta_m$, a Gaussian measure $\mcl{N}(0, \mcl{R})$
and a Radon measure $\nu$ identified via $\hat{\nu}(\varrho) = \int_X \Psi(u,
\rho) \dd \lambda (u)$ so that 
\begin{equation}\label{levy-khintchine-decomposition}
\mu = \delta_m \ast \mcl{N}(0, \mcl{R}) \ast \nu.
\end{equation}
\end{theorem}
The L{\'e}vy-Khintchine representation implies that the triple $(m, \mcl{R},
\lambda)$ completely identifies an ID measure $\mu$ and so we use the
shorthand 
notation $\mu =\text{ID}(m, \mcl{R}, \lambda)$. To gain more insight into the implications of the
L{\'e}vy-Khintchine representation we recall the class of compound
Poisson random variables and their corresponding probability measures.

\begin{definition}[Compound Poisson probability measure
  {\cite[Sec.~5.3]{linde}}] \label{def-compound-poisson}
 ~Let
$\eta$ be a Radon probability measure on a Banach space $X$ and
suppose that $\{u_k\}$ is a sequence of i.i.d. random variables so that
$u_k \sim \eta$. Also,
  let $\tau$ be an independent Poisson random variable with rate $c > 0$ 
taking values in $\mbb{Z}_+$. Then 
$
u = \sum_{k=0}^\tau u_k
$
is distributed according to a compound Poisson probability measure
denoted by $\text{CPois}(c, \eta)$.
\end{definition}

It is straightforward to check that the characteristic function of a
compound Poisson measure has the form 
$$
\widehat{\text{CPois}(c, \eta)}(\varrho) = \exp \left( c \int_X (\exp( i \varrho (u)) -
1) \:
\dd \eta(u) \right) \qquad \forall \varrho \in X^\ast.
$$
See \cite[Prop.~5.3.1]{linde} for a proof of this formula along
with the fact that $\text{CPois}(c,\eta)$ is a Radon measure on $X$.

Now let us return to the characteristic function of the
 probability measure $\nu$ that was introduced
in the L{\'e}vy-Khintchine representation \eqref{levy-khintchine-decomposition}
\begin{equation}\label{pure-jump-eq}
\begin{split}
\hat{\nu}(\varrho)
&= 
\exp\left( \int_X (\exp(i \varrho (u)) - 1)\: \dd \lambda(u) \right)
\exp\left( \int_{B_X} - i \varrho (u) \: \dd \lambda(u) \right).
\end{split}
\end{equation}
If $0<\lambda(X) <\infty$ then $\lambda$ can be renormalized to
define a probability measure  $\tilde{\lambda} :=
\frac{1}{\lambda(X)} \lambda$.
 Furthermore, we can define an element $u_\lambda \in X$ so that 
$$
\varrho(u_\lambda) = -\int_X \varrho(u) \mb{1}_{B_X}(u) \dd \lambda(u) \qquad
\forall \varrho \in X^\ast. 
$$
 Putting
these observations together with
\eqref{pure-jump-eq} gives the decomposition
\begin{equation}\label{pure-jump-decomp}
\nu = \text{CPois}(\lambda(X), \tilde{\lambda}) \ast \delta_{u_\lambda}.
\end{equation}

Therefore, from
\eqref{levy-khintchine-decomposition} 
we deduce that any measure
$\mu= \text{ID}(m, \mcl{R}, \lambda)$ with $\lambda(X) <
\infty$ can be decomposed as

\begin{equation}\label{ID-decomposition-centered}
\mu= (\delta_{m + u_\lambda}) \ast
\mcl{N}(0, \mcl{R}) \ast \text{CPois}(\lambda(X), \tilde{\lambda}).
%
\end{equation}
In the remainder of this article we will restrict our attention to the
case of ID
measures with $\lambda(X) < \infty$.
Since the tail behavior of prior measures is of importance to our
well-posedness results in
Section \ref{sec:well-posedness} we now present some results
concerning the tail behavior of ID measures. We begin with 
the notion of a submultiplicative function. 

\begin{definition}[Submultiplicative function]
  A non-negative, non-decreasing and locally bounded
function $h:\reals \mapsto \reals^+$ is called submultiplicative if it
  satisfies 
$$
h(t + s) \le  C h(t) h(s) \qquad \forall t,s \in \reals
$$
with an independent constant $C >0$
\end{definition}

Our interest in the class of submultiplicative functions arises from the next theorem
that describes some of the properties of this class. 
\begin{theorem}[{\cite[Prop.~25.4]{sato}}] \label{submultiplicative-functions}
 \myhl{ \begin{enumerate}[(i)]
  \item The product of two submultiplicative functions is also
    submultiplicative. 
\item If $h$ is submultiplicative then so is $(h(a t + b))^\alpha$ for
  constants $a, b\in \reals$ and $\alpha>0$.
\item The functions $\max\{ 1, |t|\}$ and $\exp( |t|^\beta)$ for
  $\beta \in (0,1)$ are submultiplicative.
  \end{enumerate}}
\end{theorem}

We now present a theorem that relates the tail behavior of an ID
measure to that of its L{\'e}vy measure.
 This result was originally proven by Kruglov \cite{kruglov} for
ID random variables on $\reals$ with L{\'e}vy measures that are not necessarily finite. \myhl{Different generalizations 
of this result to larger spaces are also available in the literature. For example, see
\cite[Thm.~25.3]{sato} for extension to $\reals^n$ and \cite[Prop.~6.9]{peszat}
for Hilbert space valued L{\'e}vy processes. For the reader's convenience
 we briefly prove this result for Banach space valued ID random variables with 
 finite L{\'e}vy measures.}    

\begin{theorem}\label{g-moment-of-ID-measure}
Let $X$ be a Banach space and $\lambda$ be a 
L{\'e}vy measure so that $0<\lambda(X) < \infty$. Suppose that $u \sim
\mu =
  \text{ID}(m, \mcl{R}, \lambda)$,
$\mu(X) = 1$ and $\| \cdot \|_X < \infty$ $\mu$-a.s. Then, given a
submultiplicative function $h$ we have that
$h(\| \cdot \|_X) \in L^1(X,
\mu)$ if $h(\|\cdot\|_X) \in L^1(X, \lambda)$.
\end{theorem}
\begin{proof}
Let $u \sim \mu$, then following Theorem
  \ref{levy-khintchine-representation} and the decomposition \eqref{ID-decomposition-centered} above, we
  know that there exists
an element $\tilde{m} \in X$ and 
  independent random variables $w \sim \mcl{N}(0, R) $ and $v \sim \text{CPois}(\lambda(X),
  \frac{1}{\lambda(X)} \lambda)$ so that 
$$
\EE h( \|u\|_X ) = \EE h( \| \tilde{m} + w + v \|_X) \le C^3 h(\| \tilde{m}\|_X) (\EE
h(\| w\|_X) ) (\EE h(\| v\|_X))
$$
where the inequality follows because of the triangle inequality and
the fact that $h$ is non-decreasing and locally bounded. Now by \cite[Lem.~25.5]{sato}
we know that there exist constants $a,b > 0$ such that $h(x) \le b
\exp( a|x|)$. Using this bound with the assumption that $\| u\|_X
<\infty$ $\mu$-a.s. along with  Fernique's theorem \cite[Thm.~2.8.5]{bogachev-gaussian} for Gaussian
measures on Banach spaces implies that $\EE h(\| w\|_X) <
\infty$. Now suppose that $\frac{1}{\lambda(X)} \int_X h(\| u\|_X) \dd
\lambda(u) = U < \infty$. Then using the law of total expectation
\cite[Thm.~34.4]{billingsley}, the fact that $h$ is
submultiplicative and $v$ is a compound Poisson random variable we get
$$
\begin{aligned}
\EE h(\|v\|_X)  &= \EE \left( \left. \EE h \left( \left \| \sum_{k=0}^N v_k \right\|_X
\right) \right| N \right) 
\le \EE \left( \left. \EE h \left( \sum_{k=0}^N\left\| v_k \right\|_X
    \right)\right| N \right) \\
&\le \EE \left( \left. C^N\EE  \left( \prod_{k=0}^Nh\left\| v_k \right\|_X
    \right)\right| N \right) 
=\EE \left( \left. C^N \left( \prod_{k=0}^N \EE h\left\| v_k \right\|_X
    \right)\right| N \right)  \\ & = \EE ( (U C)^N | N) 
= \sum_{k=0}^\infty
  \frac{e^{-\lambda(X)} ( UC \lambda(X) )^k}{k!} < \infty.
\end{aligned}
$$
 \end{proof}

By putting Theorems \ref{g-moment-of-ID-measure} and
\ref{submultiplicative-functions} together we 
immediately obtain the following corollary concerning the moments of ID measures.
\begin{corollary}\label{moments-of-ID-measure}
Suppose that $X$ is a Banach space and $\mu = \text{ID}(m, \mcl{R},
\lambda)$. If
$\lambda$ is a L{\'e}vy measure on $X$ so that $0<\lambda(X) <\infty$,
$\mu(X) = 1$ and $\| \cdot \|_X < \infty $ $\mu$-a.s.\ then $\|\cdot \|_X \in
  L^p(X, \mu)$ whenever $\| \cdot\|_X \in L^p(X, \lambda)$ for $p \in [1, \infty)$.
\end{corollary}

Another interesting case is when the L{\'e}vy measure $\lambda$ is
convex. Recall that a Radon probability measure $\nu$ on $X$ is said
to be convex whenever it satisfies
$$
\nu( \beta A + (1 - \beta) B) \ge \nu(A)^{\beta} \nu(B)^{1- \beta} 
$$
for $\beta \in [0,1]$ and all Borel sets $A$ and $B$ (see
\cite{hosseini-convex-prior, borell-convex} for more details about
convex measures). We are interested in convex measures since they
have exponential tails under mild assumptions
\cite[Thm.~3.6]{hosseini-convex-prior}. More precisely, if
$\nu$ is a convex probability measure on $X$  and $\| \cdot \|_X< \infty$ $\nu$-a.s. then
there exists a constant $0<b < \infty $ so that $\exp( b \|
\cdot  \|_X) \in L^1(X, \nu)$. \myhl{Since the exponential is a submultiplicative function, we immediately 
obtain the following corollary.}
\begin{corollary}\label{id-measures-with-convex-levy-measures}
Suppose that $X$ is a Banach space and $\mu = \text{ID}(m, \mcl{R},
\nu)$. If
$\nu$ is a convex probability measure on $X$, $\mu(X) = 1$ and $\|
\cdot \|_X < \infty $ a.s. under both $\nu$ and $\mu$ then 
there exists a constant $b >0$ so that $\exp( b \| \cdot \|_X )
\in L^1(X, \mu)$.
\end{corollary}



At the end of this section we ask whether we would obtain an ID measure if we
used a sequence of ID random variables in order to generate a
product prior. The answer to this question is affirmative and serves as
the proof of our claim that $G_{p,q}$-priors that were introduced
earlier belong to the class of ID probability measures.
\begin{theorem}\label{product-ID-prior-is-ID}
Let $X$ be a Banach space with an unconditional Schauder basis $\{ x_k
\}$ and let $\mu$ be the product prior that
is obtained from
$\{\gamma_k\} \in \ell^2$ and an i.i.d. sequence $\{ \xi_k\}$ of
real valued random variables. Suppose that $\xi_k \sim ID(0, \sigma^2,
\lambda)$
where $\sigma >0$ and $\lambda$ is a symmetric and finite L{\'e}vy measure on
$\reals$ such that $\max\{1,|\cdot|^2\} \in L^1( \reals, \lambda)$.  Then
$\mu$ is a Radon ID probability measure on $X$ with characteristic function
$$
\hat{\mu}(\varrho) = \exp \left[ - \frac{1}{2}\sum_{k=1}^\infty 
\sigma^2 \gamma_k^2  \varrho (x_k)^2 +  \sum_{k=1}^\infty\int_\reals
  (\cos(\gamma_k \varrho (x_k) t_k) - 1) \dd
  \lambda(t_k)     \right] \qquad \forall \varrho \in X^\ast.
$$
\end{theorem}
\begin{proof}
  Since $\max \{1,|\cdot|^2\} \in L^1(\reals, \lambda)$, the L{\'e}vy
  measure of the $\xi_k$ has bounded second moment and so by Corollary
  \ref{moments-of-ID-measure} we see that 
 $\VV \xi_k < \infty$. Now it follows from 
 Theorem~\ref{product-prior-ellp} that $\| \cdot \|_X<\infty$ $\mu$-a.s.
since $\{ \gamma_k \} \in \ell^2$. The fact that  $\mu$ is Radon
follows 
from 
Theorem \ref{product-prior-is-radon}.
Now we consider the characteristic function of $\mu$.  \myhl{Using the definition of 
 the
characteristic function of $\mu$  and the fact that 
$
\hat{\xi}_k(z) =  \exp\left( -\frac{1}{2}\sigma^2 z^2 + \int_\reals \cos( tz - 1) \dd \lambda(z)\right)
$, we can write
$$
\begin{aligned}
  \hat{\mu}(\varrho) &
= \prod_{k=1}^\infty \EE \exp \left( i \gamma_k \xi_k \varrho (x_k) \right)
=\exp \left[ - \frac{1}{2}\sum_{k=1}^\infty 
\sigma^2 \gamma_k^2  \varrho (x_k)^2 +  \sum_{k=1}^\infty\int_\reals
  (\cos(\gamma_k \varrho (x_k) t_k) - 1) \dd
  \lambda(t_k)     \right]. 
\end{aligned}
$$}
Now consider the sequence of measures $\{\mu_N\}_{N=1}^\infty$ that are defined via
$$
\hat{\mu}_N(\varrho) = \exp \left[ - \frac{1}{2}\sum_{k=1}^N 
\sigma^2 \gamma_k^2  \varrho (x_k)^2 +  \sum_{k=1}^N\int_\reals
  (\cos(\gamma_k \varrho (x_k) t_k) - 1) \dd
  \lambda(t_k)     \right].
$$
Each $\mu_N$ is ID given the fact that a finite sum of ID random
variables is ID. Since the $\{x_k\}$ are normalized and $\{
\gamma_k\}\in\ell^2$  then $\sum_{k=1}^\infty \gamma_k^2 \varrho (x_k)^2 <\infty$. Furthermore, using the inequality
$|\cos(t) - 1 |\le t^2$ we can write 
$$
 \sum_{k=1}^\infty\int_\reals
  (\cos(\gamma_k\varrho(x_k) t_k) - 1) \dd
  \lambda(t_k) \le  \sum_{k=1}^\infty\int_\reals
  \gamma_k^2\varrho(x_k)^2 t_k^2 \dd
  \lambda(t_k) = \sum_{k=1}^\infty
  \gamma_k^2\varrho(x_k)^2\int_\reals t_k^2 \dd
  \lambda(t_k)  
 $$
But this term is also bounded since $\{ \gamma_k \} \in \ell^2$
,
$\{x_k\}$ are normalized and $\max\{ 1, |x|^2\} \in L^1(\reals,
\lambda)$.
Then, $\hat{\mu}_N(\ell) \mapsto
\hat{\mu}(\ell)$ for all $\ell \in X^\ast$ and so the sequence $\mu_N$
converges weakly to $\mu$. Therefore, $\mu$ is also ID by \cite[Thm.~5.6.2]{linde}.
 Observe that the L{\'e}vy measure of $\mu$ is
concentrated along the coordinate axes of the basis $\{ x_k\}$.
\end{proof}

\section{Stochastic process priors on BV}
\label{sec:stochastic-process-prior}
\myhl{
Total variation regularization is a classic technique for recovery of 
blocky images in the variational approach to inverse problems \cite[Ch.~8]{vogel}. As we mentioned earlier, it was shown in \cite{lassas-can-we-use-tv} 
that the TV-prior is not discretization invariant and converges to a Gaussian measure in the limit of fine discretizations.
In this section we consider prior measures that are defined as laws of stochastic processes with jump discontinuities in order to 
model discontinuous functions with bounded variation. The resulting prior measures are defined directly on the function space $BV$ and so our definition can get around the inconsistency that was observed in \cite{lassas-can-we-use-tv}. We emphasize that our construction of a $BV$-prior does not 
disprove the result of \cite{lassas-can-we-use-tv} but provides a well defined 
alternative to the classic TV-prior. We also note that 
our approach is not the only way to construct a well defined TV-prior. For example, \cite{yao-tv-gaussian} presents a TV-prior that is 
absolutely continuous with respect to an underlying Gaussian measure 
and results in a well-posed inverse problem. 
 
Following \cite[Ch.~13]{leoni} we define the space $BV(\Omega)$ of functions of bounded variation
on an open set $\Omega \subset \reals^n$  as the space of 
functions $u\in L^1(\Omega)$ whose first order partial derivatives are finite signed Radon measures i.e. 
for $j= 1, 2, \cdots, n$ there exist finite signed measures $\mcl{M}_j: \mcl{B}(\Omega) \mapsto \reals$ so that 
$$
\int_\Omega u(t) \frac{\partial \phi}{\partial t_j}(t) \dd \Lambda(t) =    - \int_\Omega \phi(t) \dd \mcl{M}_j(t), \qquad \forall \phi \in C^\infty_c(\Omega).
$$
We define the {\it variation} of $u$ as 
$$
V(u) := \sup \left\{ \sum_{k=1}^n \int_\Omega \phi_k(t) \dd \mcl{M}_k(t) ~ \bigg|~ \phi \in C^\infty_c(\Omega; \reals^n), |\phi(t)| \le 1, \forall t \in \Omega  \right\},
$$
where $C^\infty_c(\Omega;\reals^n)$ denotes the space of vector valued smooth functions with compact support in $\Omega$. The space $BV(\Omega)$
is a Banach space when equipped with the norm 
$$
\| u\|_{BV(\Omega)} := \| u\|_{L^1(\Omega)} + V(u)
$$
but it
is not separable \cite[Prop.~2.3]{buttazzo}. There is a correspondence between the space $BV([0,1])$ and  
the space of functions with finite total variation in one dimensions. Recall that the {\it total variation} of a function $u: [0,1] \mapsto \reals$ is defined as 
$$
TV(u) := \sup \left\{ \sum_{k=1}^K |u(t_k) - u(t_{k-1}) |  \right\}
$$
where the supremum is taken over all finite partitions $0 =t_0 < t_1 < t_2 < \cdots < t_{K} =1$ of the interval $[0,1]$ with $K \in \integers$. It is known 
that if $TV(u) < \infty$ then $V(u) \le TV(u)$ and every $u\in BV([0,1])$ has a right continuous representative with bounded total variation 
\cite[Thm.~7.2]{leoni}. We prefer to work with the space $BV$ and its corresponding norm rather than the total variation functional 
since the former is readily defined in higher dimensions.
We start by  constructing a prior measure on $BV([0,1])$
as the law of a pure Jump L{\'e}vy process. We shall generalize our construction to  
$BV(\Omega)$ 
 later in Section~\ref{sec:bv-prior-high-dim}. 

In order to define our prior measure we will use some well-known results from the theory of L{\'e}vy processes (see \cite{tankov} for an extensive introduction).
 Using
the L{\'e}vy-Khintchine formula for L{\'e}vy processes \cite[Thm.~3.1]{tankov} we identify a L{\'e}vy process $u(t)$ via its characteristic function 
$$
\EE \exp( i s u(t) ) = \exp(t \psi(s))  \qquad \text{for} \qquad s \in \reals,
$$
where 
$$
\psi(s) =  i m s  - \frac{1}{2} (\sigma s)^2 + \int_\reals \exp( i \xi s) - 1 - i s \xi \mb{1}_{\{|\xi| \le 1\}}(\xi)  \dd \lambda(\xi).
$$
Here, the constants $m\in \reals$ and $\sigma \ge 0$ are fixed and $\lambda$ is a L{\'e}vy measure on $\reals$ (see Definition~\ref{levy-measure-definition}). Similar to the case of ID measures, the {\it characteristic triplet} $(m, \sigma, \lambda)$ uniquely identifies the 
stochastic process $u(t)$.  Certain pathwise properties of $u(t)$ can be inferred from its characteristic 
triplet. 
\begin{theorem}[{\cite[Prop.~3.9]{tankov}}]\label{pure-jump-levy-process}
 \myhl{ Let $u(t)$ be a L{\'e}vy process with characteristic triplet $(m, \sigma, \lambda)$. Then $u(t) \in BV([0,1])$
a.s. and $\EE \| u \|_{BV([0,1])} < \infty$ if 
\begin{equation}\label{bounded-variation-condition}
\sigma = 0 \quad \text{and} \quad  \int_{\{ |\xi| \le 1\}} |\xi| \dd \lambda(\xi) < \infty. 
\end{equation}
Such a process is of the pure jump type. 
If in addition $\lambda(\reals) < \infty$ then $u(t)$ is a compound Poisson process with piecewise constant 
sample paths.} 
\end{theorem}

 Thus, the law of a L{\'e}vy process $u(t)$ that satisfies \eqref{bounded-variation-condition}
coincides with a probability measure that is supported on $BV([0,1])$. Let us denote this measure by $\mu$. We wish  to use 
this measure as a prior within the Bayesian framework and achieve a well-posed inverse problem. An 
important question at this point is whether $\mu$ is a Radon measure on $BV([0,1])$ since the Radon 
property can often simplify the well-posedness analysis. We will now show that in the compound 
Poisson case, i.e. when $\lambda(\reals) <\infty$, the measure $\mu$ is tight and hence Radon \cite[Lem.~12.6]{aliprantis}. To our knowledge this result does not hold for general choices of the L{\'e}vy 
measure $\lambda$.

Recall Helly's selection principle \cite[Thm.~12]{hanche} stating that 
a set $A \subset BV([0,1])$ is relatively compact if there exists a constant 
$M > 0$ so that 
$$
\| w\|_{L^\infty([0,1])} < M, \qquad TV(w) < M, \qquad \forall w \in A. 
$$
Thus, 
 to show that ${\mu}$ is a tight measure on $BV([0,1])$ we need to argue 
that for every $\epsilon > 0$ there is an $M > 0$ so that 
$$
{\mu}( \{ w \in BV([0,1]) : \| w\|_{L^\infty([0,1])} \ge M, \:
TV(w) \ge M \} ) < \epsilon. 
$$
Now suppose that $u(t)$ is a compound Poisson process with characteristic triplet $(0, 0, \lambda)$ such that 
$c = \lambda(\reals) < \infty$ and $\int_\reals s \dd \lambda(s) < \infty$ (We added the last condition to ensure that $\lambda$ can be normalized to define a probability measure with 
bounded expectation).
 Then, we can write
\begin{equation}\label{compound-Poisson-process-generic}
u(0) = 0, \quad u(t) = \sum_{k=1}^{\tau(t)} \xi_k \quad \text{for} \quad t \in (0,1].
\end{equation}
Here $\tau(t)$ is a Poisson process with  intensity $c$ i.e. 
$$
\PP (\tau(t) = k) = \frac{(ct)^k}{k!} \exp( - ct)
$$
and $\{ \xi_k\}$ is an i.i.d.\ sequence of random variables distributed according 
to the measure $c^{-1}\lambda$. A few draws from such a process are given in Figure~\ref{fig:stochastic-process-prior-sample}(a)  when $\lambda$ is a standard Gaussian. 
Using this representation of $u(t)$ and the law of total expectation \cite[Thm.~34.4]{billingsley} we can write 
\begin{equation}\label{uniform-norm-poisson-process}
\EE \| u\|_{L^\infty([0,1])} =  \EE \sup_{t\in [0,1]} \left| \sum_{k=1}^{\tau(t)} \xi_k \right| \le \EE\left( \EE \sum_{k=1}^{N} |\xi_k| \bigg| \tau(1) = N \right) = c \EE |\xi_1| < \infty. 
\end{equation}
Furthermore, since the total variation of a piecewise constant function is simply the 
sum of the  the jump sizes we have 
\begin{equation}\label{total-variation-poisson-process}
\EE TV(u) = \EE \sum_{k=1}^{\tau(1)} |\xi_k | = c \EE |\xi_1| < \infty. 
\end{equation}
Now it follows from Markov's inequality that for any
 $M >0$
$$
\PP( \| u\|_{L^\infty([0,1])} \ge M) \le \frac{\EE \| u\|_{L^\infty([0,1])}}{M},\qquad 
\PP( TV(u) \ge M) \le \frac{ \EE TV(u)}{M}.
$$  
A straightforward argument yields that for any choice of $\epsilon >0$ we can choose 
$M > 0$ large enough 
so that $\PP (\| u\|_{L^\infty([0,1])} > M , TV(u) > M ) \le \epsilon$ and so 
the measure $\mu$, the law of $u(t)$, is tight (Radon) on $BV([0,1])$.}

\subsection{Combination with Gaussian processes}
\myhl{The compound Poisson process is a convenient model for functions with jump discontinuities. However, the fact that its sample paths  are piecewise constant can be too restrictive. In order to achieve a more flexible prior,
 that can model piecewise continuous functions, we combine our compound Poisson processes 
with a Gaussian process. If the sample paths of the Gaussian process 
 are sufficiently regular
 then the resulting prior measure will still be concentrated on $BV([0,1])$.  
The theory of Gaussian processes is well 
developed and a detailed introduction can be found in the monograph \cite{rasmussen}. 
Here we recall some basic results  and only consider the case of a Gaussian process with $C^\infty$ sample paths. Our approach can easily be 
generalized to less regular Gaussian processes by choosing a different kernel \cite[Sec.~4.2]{rasmussen}.

 Let $g(t)$ be a random function on $[0,1]$ so that
for any finite collection of points $\{ t_{k} \}_{k=1}^n$ the random variables $\{ g(t_k) \}_{k=1}^n$ are jointly Gaussian. Furthermore, suppose that 
$$
(g(t_1), \cdots, g(t_n))^T \sim \mcl{N}( 0, \mb{K}) \qquad  \text{where} \qquad \mb{K}_{k,j} = \kappa(t_k, t_j).
$$
Here  $\kappa(r,s) := \exp( -b|r-s|^2)$ is the {\it covariance kernel} of $g(t)$
and $b > 0$ is a fixed constant. Under these assumptions $g(t)$ is a mean zero Gaussian process and its samples are almost surely in
$C^\infty([0,1])$ (see \cite[Sec.~2.5.4]{paciorek-dissertation}). By definition, the Law of $g(t)$ is a Gaussian measure and since the kernel 
$\kappa$ is positive definite and continuous it follows from the Karhunen-Lo\'{e}ve
theorem (see \cite[Sec.~2.3]{ghanem}  that the law of $g(t)$ is supported on a 
Hilbert space and so it is a Radon measure.   

Now let us consider a compound Poisson process $u(t)$ as in 
\eqref{compound-Poisson-process-generic} in addition to the Gaussian process $g(t)$. Then 
the new process 
\begin{equation}\label{piecewise-smooth-process}
v(t) = g(t) + u(t)
\end{equation}
will have sample paths that are piecewise $C^\infty$ with finitely many jumps. Furthermore, since 
the laws of $u(t)$ and $g(t)$ are both Radon then the law of $v(t)$ is also Radon. Examples of draws from the 
process $v(t)$ are given in Figure~\ref{fig:stochastic-process-prior-sample}(b)
}
\begin{figure}[htp]
  \centering
        \raisebox{.18\textwidth}{a)}
        \includegraphics[width=.4\textwidth]{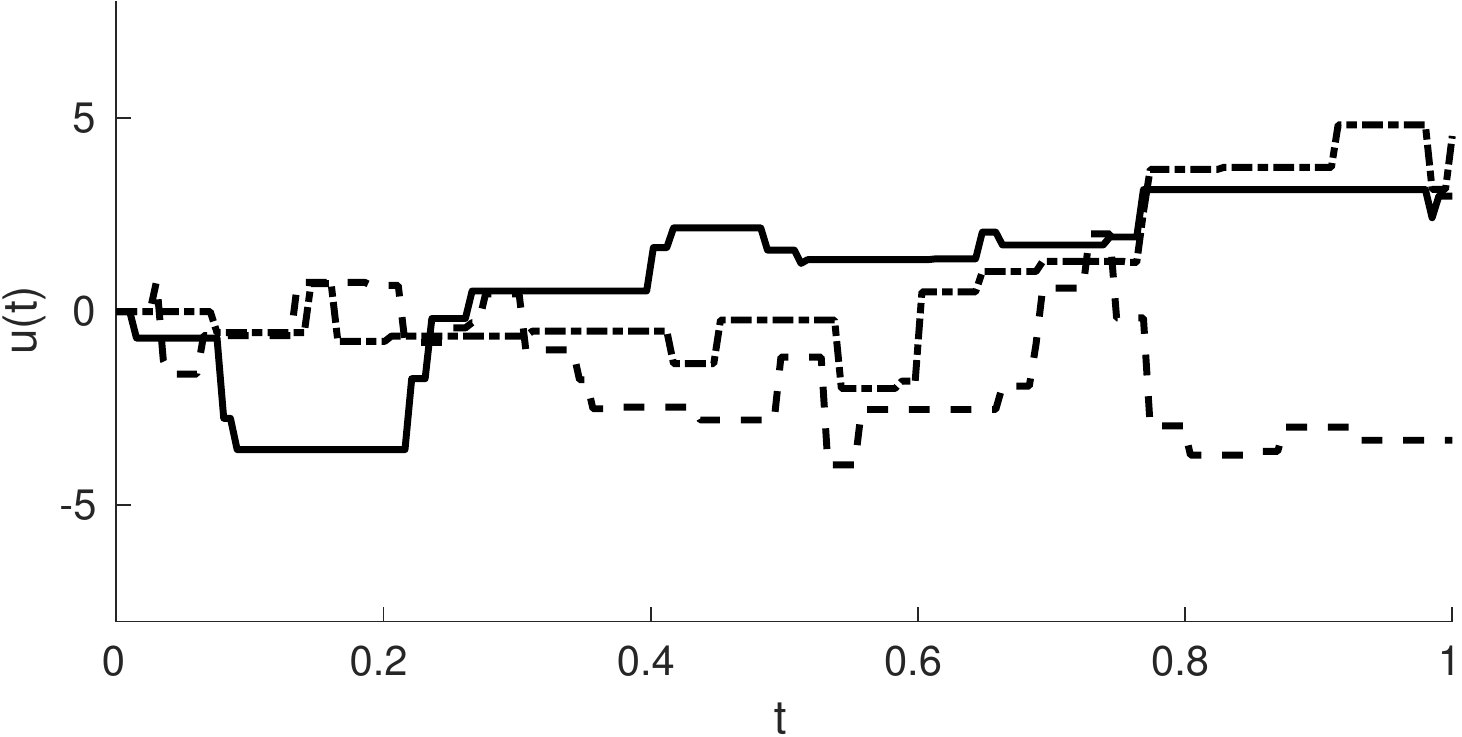} \qquad
        \raisebox{.18\textwidth}{b)}
        \includegraphics[width=.4\textwidth]{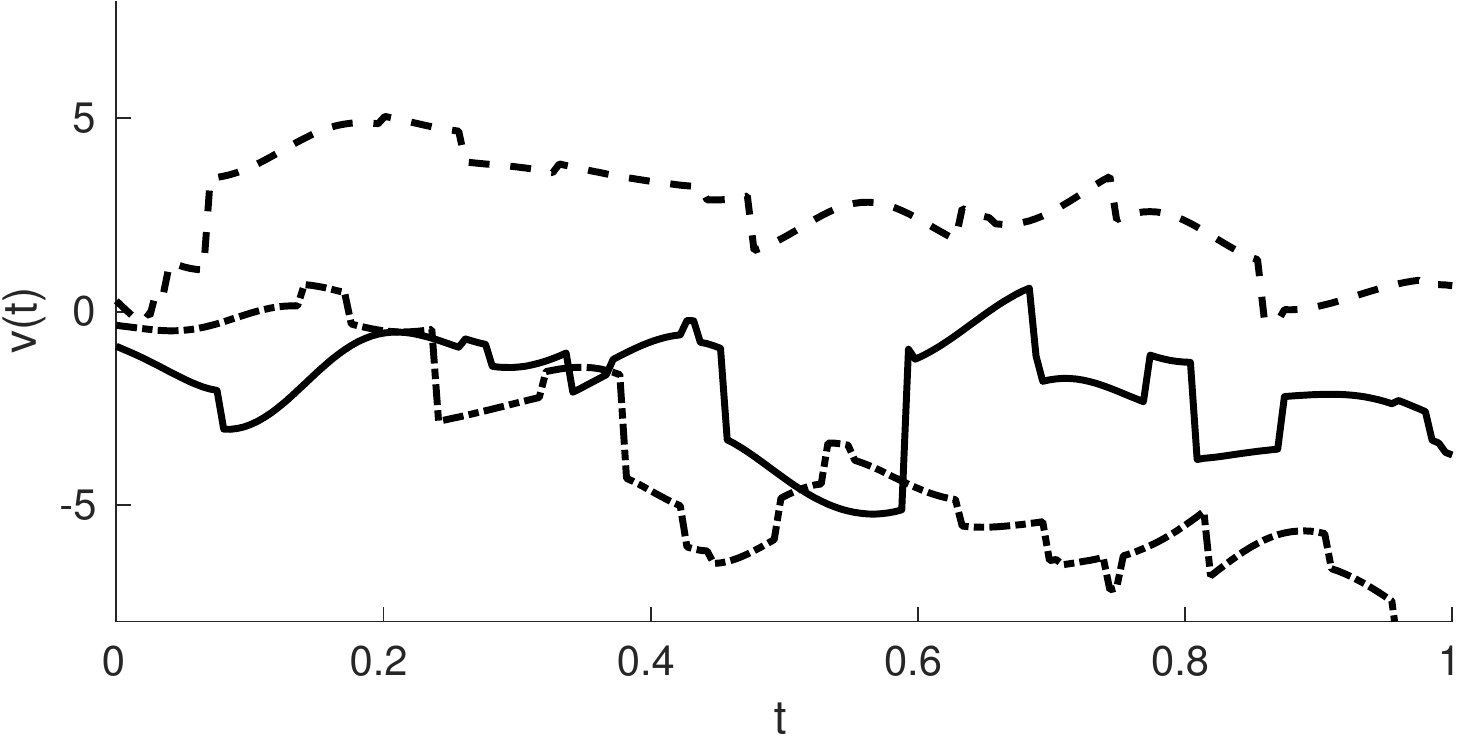} \\

        \raisebox{.3\textwidth}{c)}
        \includegraphics[width=.4\textwidth, clip=true, trim= 1cm 0cm 2cm 0 cm]{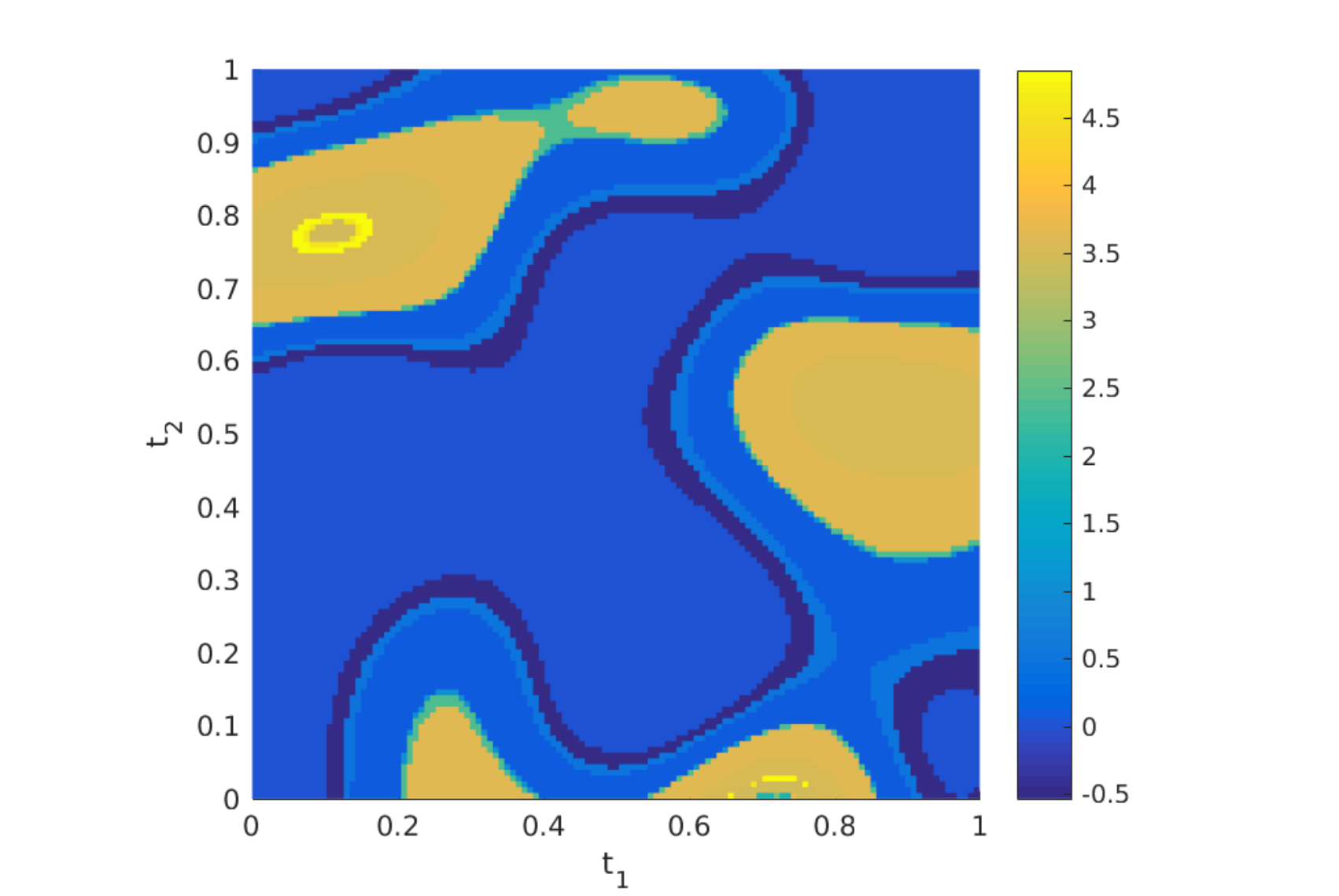} \qquad{ }
        \raisebox{.3\textwidth}{d)}
        \includegraphics[width=.38\textwidth, clip=true, trim= 1cm 0cm 2cm 0 cm]{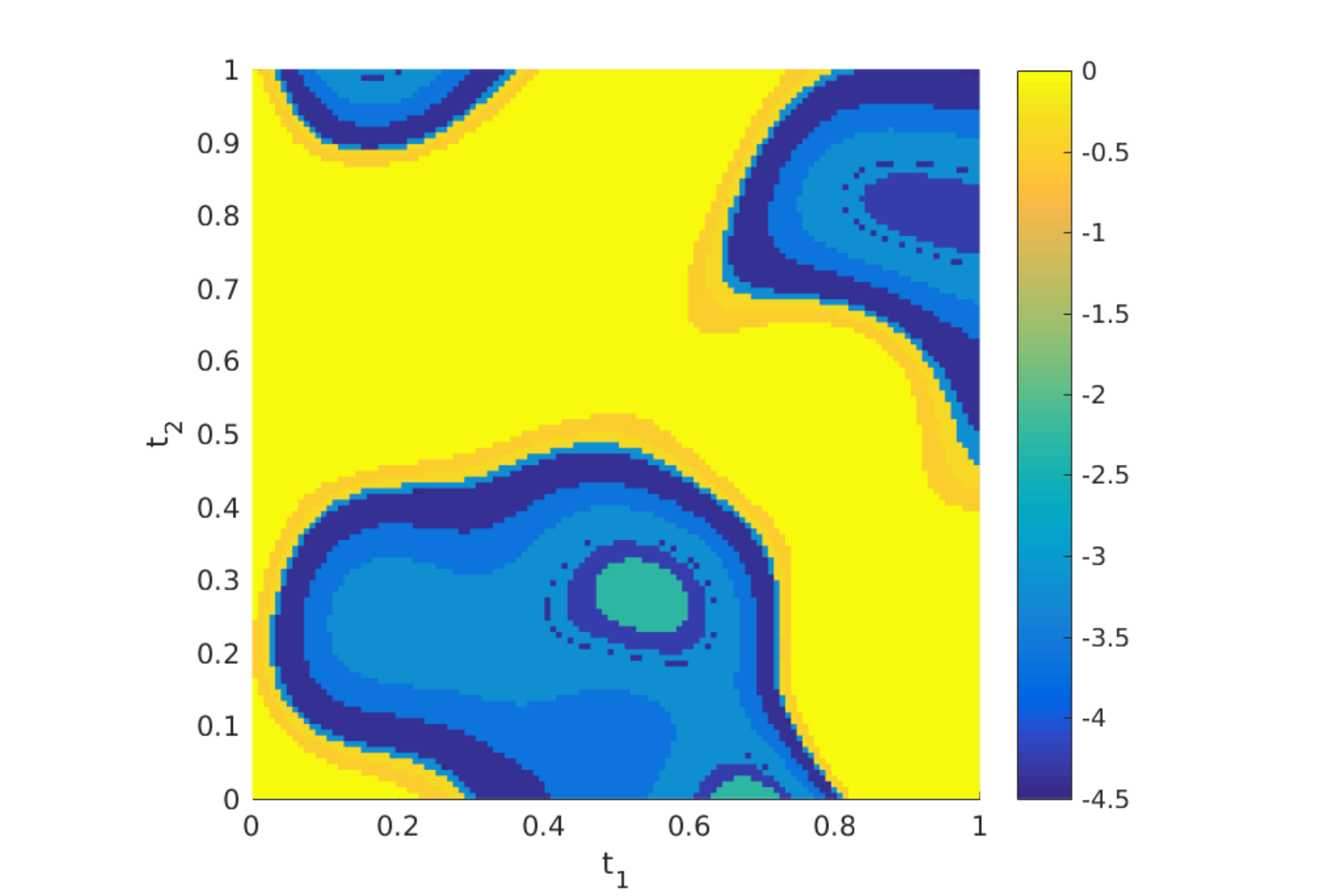}
  \caption{(a) Three samples from the compound Poisson process \eqref{compound-Poisson-process-generic} with normal jumps depicting the 
 piecewise constant sample paths. (b) Samples from the $v(t)$ process generated using 
\eqref{piecewise-smooth-process} by combining a smooth Gaussian process and an independent 
compound Poisson process with normal jumps. (c,d) Two draws from the random field $u(\mb{t})$
generated using \eqref{BV-process-in-2D-example} by combining a smooth Gaussian field with a 
Poisson process. }
  \label{fig:stochastic-process-prior-sample}
\end{figure}

\subsection{Extension to higher dimensions}\label{sec:bv-prior-high-dim}
\myhl{At the end of this section we discuss some possibilities for extension of the 
 compound Poisson process priors to random fields in $BV(\Omega)$ for compact domains $\Omega \subset \reals^n$, for 
$n=2, 3$, with Lipschitz boundary. Let $g$ be a Gaussian process on $\Omega$ with kernel 
$\kappa(r, s) = \exp( -|{r- s}|^2)$ i.e.
$$
(g({t}_1), \cdots, g({t}_n))^T \sim
\mcl{N}(0, \mb{K}), \quad \mb{K}_{k,j} = \kappa( {t}_k, {t}_j)
\quad \text{for any collection} \quad {t}_1, \cdots {t}_n \in \Omega.
$$    
Under these assumptions $g\in C^\infty(\Omega)$ a.s. Now consider the random field
\begin{equation}\label{BV-process-in-2D-example}
u({t}) = \sum_{k=1}^{\tau(g^+({t}))} \xi_k, \qquad  \text{for } t\in \Omega \qquad \text{where } g^+(t) := \max\{0, g(t)\}.
\end{equation}
As before, $\tau$ is an independent Poisson random variable with rate $c > 0$ and $\{\xi_k\}$ is a sequence of i.i.d. random variables distributed according to the probability measure $c^{-1} \lambda$. It is straightforward to check that samples from $u$ are piecewise 
constant functions on $\Omega$ that jump along a finite number of the positive level sets of the field $g$. These level sets are chosen by the poisson random variable $\tau$.
 Since we assumed that $g$ is 
in $C^\infty(\Omega)$ a.s. we expect that the level sets of $g$ are also smooth and so $u$ is a piecewise 
constant function that jumps along finitely many smooth curves (surfaces). We will now check whether the 
law of the field $u$ is indeed supported on $BV(\Omega)$.

In what follows we will occasionally suppress the dependence of different functions on $t$ to make the 
expressions more readable.
Consider a test function $\phi \in C^\infty_c(\Omega; \reals^n)$ so that $|\phi(t)| \le 1$.  We can write
$$
\int_{\Omega} u \: \text{div} \phi  \: \dd \Lambda(t) = \int_\Omega \left( \sum_{k=1}^{\tau(g^+)} \xi_k\right) \: \text{div} \phi \: \dd \Lambda(t) = \sum_{j=1}^{\tau( \max g^+ )} \int_{A_j} \left(\sum_{k=1}^j \xi_k \right) \: \text{div} \phi \: \dd \Lambda(t)  
$$
with the convention that the sum inside the integral is set to zero whenever $\tau(\max g^+) = 0$.
The sets $
A_j:=  \{ t \in \Omega \:\big| \: \tau(g^+(t)) = j\}
$
are the subsets of $\Omega$ on which $u$ is constant. Since $\tau$ is almost surely finite, we can 
integrate by parts \cite[Thm.~5.6]{evans-measure-theory} to get 
$$
\int_{\Omega} u \: \text{div} \phi  \: \dd \Lambda(t) = \sum_{j=1}^{\tau(\max g^+)} \int_{\partial A_j} \mcl{J}_j(u) \langle \phi, \vartheta_j \rangle \dd \Theta_j^{n-1}(s).    
$$
Here $\partial A_j$ is the boundary of $A_j$, $\vartheta_j$ is the unit outward normal  on $\partial A_j$, 
$\mcl{J}_j(u)$ is the jump of  $u$ across $\partial A_j$ going from $A_{j-1}$ to $A_j$ and $\Theta_j^{n-1}$ are the $(n-1)$-dimensional Hausdorff measures on $\partial A_j$ \cite[Ch.~2]{evans-measure-theory}. Recall that the 
$(n-1)$-dimensional Hausdorff measure of a simple curve in 2D (resp. surface in 3D) coincides with its length (resp. surface area) \cite[Sec.~3.4]{evans-measure-theory}. Since the size of the jumps of $u$ are a.s. finite we can write 
$$
\begin{aligned}
  \int_\Omega u \: \text{div} \phi \: \dd \Lambda(t) &\le \left(
    \sum_{k=1}^{\tau(\max g^+)} |\xi_k| \right) \left(
    \sum_{j=1}^{\tau(\max g^+)} \int_{\partial A_j} \langle \phi,
    \vartheta_j \rangle \dd \Theta_j^{n-1}(s) \right) \\ & \le \left(
    \sum_{k=1}^{\tau(\max g^+)} |\xi_k| \right) \sum_{j=1}^{\tau(\max
    g^+)}\Theta_j^{n-1}({\partial A_j})
\end{aligned}
$$
 The main question now is whether or not $\Theta_j^{n-1}( \partial A_j)$ are finite a.s. This is solely a property 
of the field $g$. We need a generalization of Rice's formula in order to respond to this question.

\begin{theorem}[{\cite[Thm.~6.8 and Prop.~6.12]{azais}}]\label{gaussian-level-set-thrm}
  \myhl{Let $\Omega$ be a compact set in $\reals^n$ with $n \ge 1$. Let $z: \Omega \mapsto \reals$ be a Gaussian random field so that $z \in C^2(\Omega)$ a.s. and
$\VV z(t) >0$ for all $t \in \Omega$. For a fixed constant $b \in \reals$, $ \EE {\Theta}^{n-1}( \{ t\in \Omega | z(t) = b\} ) <\infty$ 
if the pair $(z(t), \nabla z(t))$ have a  joint density on $\Omega \times \reals^n$ that is locally bounded.}
\end{theorem}

It is well known \cite[Thm.~2.2.2]{adler-geometry} that for $j \in \{ 1,\cdots, n\}$ the processes $g_j(t) := \frac{\partial g}{\partial t_j}(t)$ are themselves Gaussian processes with kernels $\kappa_j(r,s) = \frac{\partial^2}{\partial r_j \partial s_j} \kappa(r,s)$ whenever the second order partial derivative of the kernel exists. Using this fact it is straightforward to check that our choice of the field $g$ satisfies the 
assumptions of the above theorem. Therefore, $\Theta_j^{n-1}(\partial A_j)$ are finite a.s. and we conclude that 
$V(u) < \infty$ a.s. 
 }

 \myhl{We show two samples from the random field of \eqref{BV-process-in-2D-example} on the box $[0,1]^2$ in Figure~\ref{fig:stochastic-process-prior-sample}(c--d) with standard normal jumps. The choice of the field $g$ influences the shape of the discontinuity curves  of $u$. The main difficulty in proving  $u \in BV(\Omega)$ is in showing that the level sets of the underlying Gaussian field have finite length. Theorem~\ref{gaussian-level-set-thrm} allows us to relax the regularity assumptions 
on $g$ and take Gaussian fields that are in $C^2(\Omega)$ rather than $C^\infty(\Omega)$ but for less regular fields it is not clear whether the level sets have finite length. Finally, one can show that the law of the process in \eqref{BV-process-in-2D-example} is a Radon measure on $BV(\Omega)$. The method of proof is identical to our 
argument for the 1D case except that now we need a different version of Helly's selection principle \cite[Thm.~5.5]{evans-measure-theory} in order to construct compact sets in $BV(\Omega)$.}



 \section{Well-posed Bayesian inverse problems}\label{sec:well-posedness}
\myhl{ Recall from Section~\ref{sec:introduction} that we are interested in the problem of inferring 
a parameter $u \in X$ from data $y \in Y$ that are related via a generic stochastic mapping $\tilde{\mcl{G}}$ that models the physical process that generates the data as well as the measurement noise:
$$
y = \tilde{\mcl{G}}(u).
$$
In order to solve this problem we employ Bayes' rule \eqref{bayes-rule}.
In this section we collect certain conditions on the prior measure $\mu_0$ and the likelihood potential $\Phi$ that result in well-posed inverse problems. We consider a general enough setting that
 encompasses the heavy-tailed priors of Section~\ref{sec:non-gaussian-priors} and the stochastic process priors of 
Section~\ref{sec:stochastic-process-prior}. 
 We assume that the 
parameter space $X$ is a Banach space that is not necessarily separable (such as $BV$) and 
the prior measure $\mu_0$ is possibly heavy-tailed (such as the $G_{p,q}$-priors) and not necessarily Radon (such as the law of the
pure jump L{\'e}vy processes when  $\lambda$ is 
not finite).}

\myhl{The main results of this section are Theorems \ref{existence-uniqueness} and \ref{stability-TV}
that establish the existence, uniqueness and stability of the posterior measure. 
We acknowledge that these theorems are very similar to the results
 in \cite[Sec.~4.1]{stuart-bayesian-lecture-notes} and \cite{sullivan}. In comparison to these articles, we impose 
slightly different assumptions on the potential $\Phi$ and assume that the space $X$ is not necessarily separable.} \myhl{We also note that 
under the assumption that the prior measure $\mu_0$ is a Radon measure one can immediately generalize the 
result of \cite{stuart-bayesian-lecture-notes} to non-separable parameter spaces $X$
 using the fact that 
Radon measures on a Banach space are automatically concentrated on a separable subspace.}
\begin{theorem}[{\cite[Thm.~7.12.4]{bogachev2}}] \label{concentration-of-measure}
Let $\mu$ be a Radon probability measure on a Banach space $X$. Then, there exists
a reflexive and separable Banach space $E$ embedded in $X$ such that $\mu(X \setminus E) =0$ and the closed balls of $E$ are compact in $X$. 
\end{theorem}

\myhl{It is important to note that while this theorem guarantees the existence of the separable space 
$E$ it does not provide us with a method for identifying $E$ or its norm. In the case of the product priors of 
Section~\ref{sec:non-gaussian-priors} one can argue that the measures are concentrated on a separable Hilbert space
but for the stochastic process priors of Section~\ref{sec:stochastic-process-prior} it is no longer clear 
what the space $E$ is and so it is more convenient for us to analyze the inverse problem on the 
ambient space $X$ rather than passing to the space $E$.}

 \myhl{We will present our well-posedness results using the total variation metric, since this metric is less often used in previous works,
 and refer the reader to \cite{stuart-bayesian-lecture-notes} for proofs using the Hellinger metric that can easily be 
generalized to our setting by comparison to the proofs using the total variation metric. Given the inequalities \eqref{tv-hellinger-equivalence}
we immediately see that well-posedness in one of these metrics implies well-posedness in the other but the convergence rates 
will differ.}

We start by presenting
minimal assumptions on the likelihood potential and the forward map and
make our way to more specific cases of inverse problems such as
problems with linear forward maps. In a nutshell, as we put more restrictions on
$\Phi$ we are able to relax our assumptions on $\mu_0$. In order to help with navigation through this section we present Table
\ref{tab:result-summary} that collects our main results and the key
underlying assumptions. 

\begin{table}[htp]
\small
  \centering
  \begin{tabular}[htp]{| m {.2 \textwidth} | m {.55\textwidth} | m {.15 \textwidth} |}
\hline
    Theorem/Corollary & Main assumptions  & type of result \\ \hline 
    Theorem~\ref{existence-uniqueness} & $\Phi$ is locally bounded and 
Lipschitz in $u$. & $\mu^y$ is well-defined  \\ \hline 
   Theorem~\ref{stability-TV} & $\Phi$ satisfies Assumption
                             \ref{assumption-on-likelihood} & $\mu^y$
                                                              depends
                                                              continuously
                                                              on $y$
    \\ \hline 
Corollary~\ref{well-posedness-with-bounded-moments} &
$\Phi$ has polynomial growth in $u$ and $\mu_0$ has finitely many
                                                  moments &
well-posedness \\ \hline 
Corollary~\ref{additive-noise-stability} & $\Phi \ge 0$ in addition to
                                        Assumption
                                        \ref{assumption-on-likelihood}
                                          & well-posedness \\ \hline 
Corollary~\ref{well-posed-linear-inverse-problem-with-ID-prior} &
$Y = \reals^m$,                                      measurement noise
                                                    is additive and
                                                    Gaussian,
          prior is ID & well-posedness  \\ \hline
Corollary~\ref{linear-inverse-problems-well-posedness} & $Y = \reals^m$,
                                                    forward map is
                                                    linear and
                                                    bounded,
                                                    measurement noise
                                                    is additive and
                                                    Gaussian &
                                                               well-posedness
    \\ \hline
Corollary~\ref{well-posedness-ellp-prior} & $Y = \reals^m$,
                                                    forward map is
                                                    linear and
                                                    bounded,
                                                    measurement noise
                                                    is additive and
                                                    Gaussian,
          $G_{p,q}$-prior & well-posedness \\ \hline 

  \end{tabular}
  \caption{Summary of the key theorems and corollaries of Section
    \ref{sec:well-posedness}. In each case we identify the key
    underlying assumptions as well as the type of final result.}
  \label{tab:result-summary}
\end{table}

We begin by
identifying some conditions on $\Phi$ that allow us to use a
very large class of prior measures including those that are heavy-tailed.

\begin{assumption} \label{assumption-on-likelihood}
Suppose that $X$ and $Y$ are Banach spaces and
the likelihood potential
$\Phi: X \times Y \mapsto \reals$
satisfies the
  following properties:  \label{forward-assumption}
  \begin{enumerate}[(i)]
  \item (Lower bound in $u$): There is a positive and
    non-decreasing function $f_1:\reals_+ \mapsto [1, \infty)$ so that $\forall r>0$, there is a
    constant $M(r) \in \reals$ such that, $$\Phi(u;y) \ge M(r) -
\log \left( f_1(\| u\|_X) \right), \qquad \forall u \in X \text{ and } \forall y \in B_Y(r).$$

    \item (Boundedness above): $\forall r>0$ there is a constant
      $K(r)>0$ such that
 $$ \Phi(u;y) \le K(r), \qquad \forall u\in B_X(r) \text{ and } \forall y \in B_Y(r). $$

      \item (Continuity in $u$): $\forall r>0$ there exists a 
constant $L(r)>0$ such that  $$ | \Phi(u_1;y) -
\Phi(u_2,y)| \le L(r) \| u_1 - u_2 \|_X, \qquad \forall u_1,u_2 \in B_X(r) \text{ and } y\in B_Y(r).$$ 

              \item (Continuity in $y$):
There is a positive and
        non-decreasing function $f_2: \reals_+ \mapsto \reals_+$ so that 
$ \forall r>0$, there is a 
constant $C(r) \in \reals$ such that  $$ | \Phi(u;y_1) -
\Phi(u,y_2)| \le C(r) f_2(\| u\|_X)  \| y_1 - y_2 \|_Y, \qquad \forall y_1,y_2 \in B_Y(r)\text{ and }\forall u \in X. $$ 
  \end{enumerate}
\end{assumption}

Our first task is to establish the existence and uniqueness of the
posterior measure. 

\begin{theorem} \label{existence-uniqueness}
Suppose $X$ is a Banach space, $\mu_0$ is a \myhl{Borel} probability
measure on $X$ and 
let $\Phi$ satisfy  Assumptions \ref{assumption-on-likelihood} (i),
(ii) and (iii) with a function $f_1 \ge 1$. If $ f_1( \| \cdot \|_X) \in
L^1(X, \mu_0)$
then the posterior $\mu^y$ given by \eqref{bayes-rule} is a well-defined
Borel probability measure on $X$. If $\mu_0$ is Radon then so is $\mu^y$.
\end{theorem}

\begin{proof}
  Our proof will closely follow the approach of
  \cite[Thm.~4.3]{sullivan}
and \cite[Thm.~2.2]{hosseini-convex-prior}. 
 Assumption \ref{assumption-on-likelihood}(iii) implies  the
continuity of $\Phi$ on $X$ which in turn
implies that $\Phi( \cdot, y) :X \mapsto \reals$
is $\mu_0$-measurable. We will now show that the normalizing constant
satisfies $0<Z(y) <\infty$ which proves that $\mu^y$ is well-defined.  
\myhl{The assertion that
$\mu^y$ inherits the Radon property of $\mu_0$ will then follow from the absolute continuity of
$\mu^y$ with respect to $\mu_0$ \cite[Lem.~7.1.11]{bogachev2}.}

Following Assumption \ref{assumption-on-likelihood}(i) we can write
$$
\begin{aligned}
  Z(y) &
  \le \int_X \exp( \log( f_1( \| u \|_X)) - M ) \dd
   \mu_0(u) = \exp(-M) \int_X  f_1( \| u \|_X )
\dd \mu_0(u) <\infty.
\end{aligned}
$$
We now need to show that the normalizing constant $Z(y)$ does not vanish. 
\myhl{It follows from Assumption \ref{assumption-on-likelihood} that for $R > 0$
$$
\begin{aligned}
  Z(y) & 
   \ge 
 \int_{B_X(R)} \exp(-K) \dd \mu_0(u)
 = \exp(-K) \mu_0( B_X(R)).
\end{aligned}
$$
 However, $\mu_0(B_X(R) ) > 0$
for large enough $R$.} \myhl{To see this consider the disjoint sets $A_k :=\{ u | k-1\le \| u\|_X <
k\}$ for $k \in \integers$. The $A_k$ are open and hence 
measurable and $\sum_{k=1}^\infty \mu_0(A_k) =
\mu_0(\bigcup_{k=1}^\infty A_k) = \mu(X) = 1$. Then the measure of at least one
of the $A_k$ has to be nonzero.}
\end{proof}


 We now establish the stability of Bayesian inverse problems with respect to 
perturbations in the data. Similar versions of the following theorems are available for Gaussian priors in \cite{stuart-acta-numerica}, for 
Besov priors in \cite{dashti-besov}, for convex priors in
\cite{hosseini-convex-prior} and for heavy-tailed priors on separable
Banach spaces \myhl{in \cite{stuart-bayesian-lecture-notes,sullivan}.} 

\begin{theorem}\label{stability-TV}
  Suppose that $X$ is a Banach space, $\mu_0$ is a \myhl{Borel} probability
  measure on $X$ and $\Phi$ satisfies Assumptions
  1(i), (ii) and (iv) with functions $f_1, f_2$. Let 
$\mu^y$ and $\mu^{y'}$ be two measures defined via \eqref{bayes-rule} for any $y$ and $y' \in Y$,
 both absolutely continuous with respect to 
$\mu_0$. \myhl{
\begin{enumerate}[(i)]
\item If $f_2(\| \cdot \|_X) f_1(\| \cdot\|_X) \in L^1(X, \mu_0)$ then
  $\forall r > 0$, there exists a constant $C(r) > 0$
   so that
  $ d_{TV} (\mu^y, \mu^{y'}) \le C \| y - y' \|_Y, \forall y,y'\in B_Y(r)  $
\item If instead $(f_2(\|\cdot\|_X))^2 f_1(\|\cdot \|_X) \in L^1(X, \mu_0)$
then $\forall r >0$ there exists 
a constant $C'(r) > 0$  so that 
$d_H (\mu^y, \mu^{y'}) \le C' \| y - y' \|_Y.$
\end{enumerate}}
\end{theorem}
\begin{proof}
\myhl{We will only prove (i) and refer the reader to \cite[Sec.~4.1]{stuart-bayesian-lecture-notes}
for the proof of (ii) that will readily generalize to our setting.}
Consider the normalizing constants $Z(y)$ and $Z(y')$. We have already
established in the proof of Theorem \ref{existence-uniqueness} that neither of these 
constants will vanish and they are both bounded. Thus the measures
$\mu^y$ and $\mu^{y'}$ are well-defined.
Applying the mean value theorem to the exponential function and using
Assumptions 1(i), (iv) 
and the assumption that $f_2(\| \cdot\|_X)f_1(\|\cdot\|_X) \in L^1(X, \mu_0)$ we obtain
\begin{equation}\label{normalizing-constant-difference}
\begin{aligned}
| Z(y) - Z(y') | &\le  \int_X \exp( -\Phi(u;y) ) | \Phi(u;y) - \Phi(u;y') | \dd \mu_0(u)  \\
&\le \left( \int_X \exp( \log(  f_1(\| u \|_X)) - M) C f_2(\|
  u\|_X) \dd \mu_0(u) \right) \| y - y' \|_Y  \\
& \le C\exp( - M) \left( \int_X f_1(\| u\|_X)f_2(\| u
  \|_X) \dd \mu_0(u) \right) \| y - y' \|_Y 
 \lesssim \| y -y' \|_Y.
\end{aligned}
\end{equation}
Following the definition of the total variation distance we have
$$
\begin{aligned}
  2 d_{TV} ( \mu^y, \mu^{y'})  &= \int_X \left| Z(y)^{-1} \exp( -
  \Phi(u;y) ) - Z(y')^{-1} \exp( - \Phi(u, y')) \right| \dd \mu_0(u)
\\ 
& \le \int_X \left| Z(y)^{-1} \exp( -
  \Phi(u;y) ) - Z(y')^{-1} \exp( - \Phi(u, y)) \right| \dd \mu_0(u) \\
& \quad + 
Z(y')^{-1} \int_X \left|\exp( -
  \Phi(u;y) ) - \exp( - \Phi(u, y')) \right| \dd \mu_0(u) =: I_1 + I_2.
\end{aligned}
$$
 Now using \eqref{normalizing-constant-difference} we have 
$$
I_1 = | Z(y)^{-1} - Z(y')^{-1} | Z(y) = \frac{Z(y)}{Z(y')} | Z(y') -
Z(y)| \lesssim  \| y - y'\|_X.
$$
Furthermore, using the mean value theorem, Assumption
\ref{assumption-on-likelihood} (i) and (iv) we can write
$$
\begin{aligned}
{Z(y')}{I_2} & =  \int_X \left|\exp( -
  \Phi(u;y) ) - \exp( - \Phi(u, y')) \right| \dd \mu_0(u) \\
& \le \int_X \exp( -
  \Phi(u;y) ) \left|  \Phi(u;y') - \Phi(u, y) \right| \dd \mu_0(u) \\
 & \le C \exp(-M) \left( \int_X \exp( \log(f_1( \| u\|_X) ))
   f_2(\| u\|_X) \dd \mu_0(u) \right) \| y - y'\|_Y 
 \lesssim \| y - y' \|_Y.
\end{aligned}
$$
\end{proof}

\myhl{The main distinction between the choice of the metrics in Theorem \ref{stability-TV} is} that in order to obtain the same rate
of convergence in the Hellinger metric we need a (possibly) stronger assumption
regarding the integrability of $f_1(\|u\|_X)$ and $f_2(\|u\|_X)$.
So far we
encountered conditions of the form $ (f_2(\| u\|_X))^p f_1(\|
u\|_X) \in L^1(X, \mu_0)$ for
$p \in \{ 0, 1, 2\}$. Intuitively, these conditions identify the interplay
between the growth of $\Phi(u;y)$ as a function of $\| u\|_X$ and the tail
behavior of the prior $\mu_0$. 

\begin{corollary}\label{well-posedness-with-bounded-moments}~
Suppose that $\Phi$ satisfies the conditions of Assumption
  \ref{assumption-on-likelihood}
with $f_1(t) = f_2(t) = \max \{1, |t|^p\}$ for $p \ge 0$ and $\mu_0$ is a \myhl{Borel}
probability measure on $X$. \myhl{If $\mu_0$ has bounded raw moments of degree up to $\lceil 2p \rceil$
then the Bayesian inverse problem \eqref{bayes-rule} is
well-posed in both the total variation and Hellinger metrics.}
\end{corollary}

\begin{corollary}~
Suppose that $\mu_0$ is a \myhl{Borel} probability measure on $X$ and 
$\exp( b \| \cdot \|_X) \in
L^1(X, \mu_0)$ for some constant $b > 0$. Then 
the Bayesian inverse problem \eqref{bayes-rule} is
well-posed in both the total variation and Hellinger metrics whenever 
$\Phi$ satisfies the conditions of Assumption
  \ref{assumption-on-likelihood} with functions $f_1, f_2$ that are
  polynomially bounded.
\end{corollary}


 For the
remainder of this section we will 
focus on specific classes of 
likelihood potentials $\Phi$ which allow us to further relax 
our assumption regarding the tail behavior of $\mu_0$. The rest of our results follow from Theorems~\ref{existence-uniqueness}
and \ref{stability-TV} but they are of great interest in practical applications. We start with
the case of additive noise models and consider linear inverse problems
afterwards. 

\subsection{The case of additive noise models}

Additive noise models have a special place in practical applications
due to their convenience and flexibility \cite{somersalo}. 
Suppose that the data is finite dimensional and, without
loss of generality, take $Y = \reals^m$,
$m \in \mbb{N}$. Now suppose that $y\in Y$ is related
to the parameter $u \in X$ via the model 
\begin{equation}\label{additive-noise-model}
y = \mcl{G}(u) + \eta \qquad \eta \sim \pi(y) \dd \Lambda(y)
\end{equation}
where $\pi(y)$ is the Lebesgue density of the measurement noise
$\eta$ and $\mcl{G}: X \mapsto \reals^m$ is the forward map. 
 It is straightforward to check that under these assumptions
\begin{equation} \label{additive-noise-likelihood}
\Phi(u;y) = -\log \pi( \mcl{G}(u) - y ).
\end{equation}
In particular if $\eta \sim \mcl{N}(0, \pmb{\Sigma})$ with an $m
\times m$ positive
definite matrix $\pmb{\Sigma}$ then 
\begin{equation} \label{gaussian-likelihood}
\Phi(u;y) = \frac{1}{2} \left\|  (\mcl{G}(u) - y) \right\|_{\pmb{\Sigma}}^2.
\end{equation}
Now if $\log \pi(y) \le 0$ (which is clearly the case when
$\eta$ is Gaussian or Laplace) then $\Phi(u;y)$ will satisfy
Assumption \ref{assumption-on-likelihood}(i) with the constant $ M 
=  0$ and $f_1(x) = 1$. This observation will allow us to relax our
assumption on the tail behavior of the prior whenever the measurement
noise is additive.

\begin{corollary} \label{additive-noise-stability}
Suppose $Y = \reals^m$, $X$ is a Banach space and $\Phi(u;y) \ge 0$ and it satisfies Assumptions
(ii) and (iv) with a function $f_2$.  Suppose that the
prior measure $\mu_0$ is a \myhl{Borel} probability measure on $X$ and
let 
$\mu^y$ and $\mu^{y'}$ be two measures defined via \eqref{bayes-rule}
for $y$ and $y' \in Y$. Then the posterior measure $\mu^y$ is
well-defined and
\begin{enumerate}[(i)]
\item If $f_2( \|
  \cdot \|_X) \in L^1(X, \mu_0)$ then $\forall r > 0, \exists C(r) > 0$ so that 
$
d_{TV}(\mu^y, \mu^{y'}) \le C(r) \| y - y' \|_Y \text{ for all } y,y' \in B_Y(r).
$
\item  If $f_2(\|\cdot\|_X) \in L^2(X, \mu_0)$ then $\exists C'(r) >0$ 
$
d_H (\mu^y, \mu^{y'}) \le C'(r) \| y - y' \|_Y.
$
\end{enumerate}
\end{corollary}

At this point it is natural to identify conditions on the distribution
of the noise and
the forward operator that guarantee that the
 likelihood potential of \eqref{additive-noise-likelihood} satisfies
 the conditions of Assumption \ref{assumption-on-likelihood}.
We will address this when $\eta$ is
Gaussian but our approach can be generalized to other types of
additive noise models.
 \begin{theorem}\label{assumption-on-forward-map}
Consider the additive noise model of \eqref{additive-noise-model} 
when $\eta \sim \mcl{N}(0, \pmb{\Sigma})$ and $\pmb{\Sigma}$ is a
positive-definite matrix.
Then the corresponding likelihood potential $\Phi(u;y)\ge 0$. Furthermore, $\Phi$
satisfies the conditions of Assumption
\ref{assumption-on-likelihood}(iv) with $f_2(x) = 1+ \tilde{f}(x)$ if there is
a positive, non-decreasing and locally bounded function $\tilde{f}:\reals_+ \mapsto
  \reals_+$ so that
\begin{enumerate}[(i)]
\item  $\exists C> 0$ for which
$
\|\mcl{G}(u)\|_2  \le C \tilde{f}( \| u\|_X), \forall u \in X.
$
\item $\forall r > 0, \exists K(r) > 0$ so that 
$
\| \mcl{G}(u_1) - \mcl{G}(u_2) \|_2 \le K(r) \| u_1 - u_2 \|_X
$ for
  all $u_1, u_2\in B_X(r)$.
\end{enumerate}
 \end{theorem}
 \begin{proof}
Since we assumed that $\eta$ is Gaussian then the likelihood potential
is of the form \eqref{gaussian-likelihood}.
Then it is clear that $\Phi(u;y) \ge 0$ which immediately implies that $\Phi$
   satisfies Assumption \ref{assumption-on-likelihood}(i) with
   $M = 0$ and $f_1(x) = 0$. Now fix $r > 0$ and suppose that $u \in B_X(r)$ and $y
   \in B_Y(r)$. Define
   $\tilde{r} = \max\{ r,  C \tilde{f}(r) \}$ and note that $\tilde{r}$ is
   bounded since we assumed that $\tilde{f}$ is locally bounded. 
 Therefore, we have
$
\Phi(u;y) \le \| \mcl{G}(u) \|_{\pmb{\Sigma}}^2 + \|
y \|_{\pmb{\Sigma}}^2  \lesssim \tilde{r}^2
$ and so
 $\Phi$ satisfies Assumption \ref{assumption-on-likelihood}(ii).

Now we will show that $\Phi$ satisfies Assumption
\ref{assumption-on-likelihood}(iii) as well. Let $r$ and $\tilde{r}$
be defined as above and consider $u_1, u_2 \in B_X(r)$ and $y \in B_Y(r)$. Using the
identity
$\| a\|_2^2 - \| b\|_2^2 =  \langle a - b , a +b \rangle $ for $a, b
\in \reals^m$ and the conditions (i) and (ii) of
the theorem we obtain
$$
\begin{aligned}
  2|\Phi(u_1;y) - \Phi(u_2;y)| 
& = \left| \langle \pmb{\Sigma}^{-1/2} (\mcl{G}(u_1) - \mcl{G}(u_2)) ,
  \pmb{\Sigma}^{-1/2} ( \mcl{G}(u_1) + \mcl{G}(u_2) - 2 y )  \rangle
\right| \\
& \le  \left(\| \mcl{G}(u_1) \|_{\pmb{\Sigma}} +\|
  \mcl{G}(u_2) \|_{\pmb{\Sigma}}
+ 2 \|y \|^2_{\pmb{\Sigma}} \right) \| (\mcl{G}(u_1) -
\mcl{G}(u_2)) \|_{\pmb{\Sigma}}  \\
&\le C(\tilde{r}) \| (\mcl{G}(u_1) - \mcl{G}(u_2) )
\|_{\pmb{\Sigma}} \le 2K(r) \| u_1 - u_2 \|_X.
\end{aligned}
$$

Finally, fix $r > 0$ and consider $y_1, y_2 \in B_Y(r)$. Then using the same line of reasoning as 
above, for any $u \in X$ we can write
$$
\begin{aligned}
2| \Phi( u;y_1)  - \Phi(u;y_2) | 
& =\left| \langle \pmb{\Sigma}^{-1/2} (y_2 - y_1) ,
  \pmb{\Sigma}^{-1/2} ( 2\mcl{G}(u) -  y_1 - y_2 )  \rangle
\right| \\
 & \le \left( \| y_2 \|_{\pmb{\Sigma}} - \| y_1 \|_{\pmb{\Sigma}} +
   2 \| \mcl{G}(u) \|_{\pmb{\Sigma}} \right) \| (y_2
   - y_1)\|_{\pmb{\Sigma}}  \\
& \le C(r) (1 + \tilde{f} (\| u\|_X)) \| y_1 - y_2 \|_2.
\end{aligned}
$$
 \end{proof}

By putting this result together with Theorem~\ref{g-moment-of-ID-measure} and Corollary~\ref{additive-noise-stability} we deduce the following 
corollary concerning the well-posedness of Bayesian inverse problems with ID priors. 

\begin{corollary}\label{well-posed-linear-inverse-problem-with-ID-prior}
  Let $X$ be a Banach space and $Y = \reals^m$. Consider the additive
  noise model:
$$
y = \mcl{G}(u) + \eta, \qquad \eta \sim \mcl{N}(0, \pmb{\Sigma})
$$
where $\pmb{\Sigma}$ is a positive definite matrix and $\mcl{G}:X \mapsto
\reals^m$ \myhl{satisfies 
the conditions of Theorem \ref{assumption-on-forward-map} with a
submultiplicative function $\tilde{f}$}. Also, suppose that $\mu_0= \text{ID}( m, \mcl{R}, \lambda)$ where
$\lambda$ is a L{\'e}vy measure such that $\lambda(X) < \infty$,
$\mu_0(X) = 1$ and $\| \cdot \|_X <\infty$ $\mu_0$-a.s. \myhl{Then the Bayesian inverse
problem \eqref{bayes-rule} is well-posed if  $1 + \tilde{f}(\| \cdot \|_X) \in L^1(X, \lambda)$.}
\end{corollary}

\subsection{The case of linear inverse problems}\label{sec:linear-inverse-problems}
 We now assume that the likelihood potential $\Phi$ has the form 
$$
\Phi(u;y) : X \times \reals^m \mapsto \reals, \qquad \Phi(u;y) = \frac{1}{2} \left\| \mcl{G}(u) - y \right\|_{\pmb{\Sigma}}^2 
$$
where $\pmb{\Sigma}$ is a positive definite matrix and $\mcl{G}:X \mapsto \reals^m$ is
bounded and linear.
This case is of particular importance due to its 
occurrence in the Compressed Sensing literature \cite{foucart} and
estimation of sparse parameters.
In this case, we can further relax our conditions on the prior measure
$\mu_0$ and achieve well-posedness 
so long as the prior $\mu_0$ has bounded first moment.
\begin{corollary}\label{linear-inverse-problems-well-posedness}
Let $X$ be a Banach space and $Y = \reals^m$. Suppose that the forward
map $\mcl{G}: X \mapsto \reals^m$ is bounded and linear and consider the
additive noise model 
$$
y = \mcl{G}(u) + \eta \qquad where \qquad \eta \sim \mcl{N}(0,
\pmb{\Sigma}) \qquad \text{and} \qquad \text{$\pmb{\Sigma}$ is positive definite}.
$$
Then the Bayesian
inverse problem of identifying the posterior $\mu^y$ via
\eqref{bayes-rule} is well-posed 
in both the Hellinger and total
variation metrics if the prior $\mu_0$ is a \myhl{Borel} probability measure
on $X$ and
$\| \cdot \|_X \in L^1(X, \mu_0)$.
\end{corollary}
\begin{proof}
 \myhl{Follows directly from
Theorems~\ref{assumption-on-forward-map} and
\ref{existence-uniqueness} and Corollary
\ref{additive-noise-stability}(i).}
\end{proof}

Let us now return to the product priors of Section
\ref{sec:priors-for-compressibility} and show that those measures result
in well-posed Bayesian inverse problems under the linear and additive
noise assumptions.

\begin{theorem}\label{well-posedness-product-prior}
Let $X$ be a Banach space with an unconditional Schauder basis $\{
x_k\}$ and take  $Y = \reals^m$. Suppose that the measurement 
noise is additive and Gaussian and the forward map $\mcl{G}$ is
bounded and linear. Furthermore, suppose that $\mu_0$ is a product
prior with sample paths 
$
u = \sum_{k=1}^\infty \gamma_k \xi_k x_k
$
where $\{\gamma_k\} \in \ell^2$ and $\{ \xi_k\}$ are i.i.d. and $\VV
\xi_k< \infty$. Then the 
 inverse problem
\eqref{bayes-rule} is well-posed in both the total variation and
Hellinger metrics.
\end{theorem}
\begin{proof}
    The fact that $\mu_0$ is a Radon probability measures on $X$ follows from Theorem
  \ref{product-prior-is-radon} and Theorem~\ref{product-prior-ellp}. Now if
  $\VV \xi_k < \infty$ then  $\EE \xi_k^2<\infty$ as well and so it follows from Theorem
\ref{bounded-moments-product-prior} that $\| \cdot \|_X \in L^2(X, \mu)$.
Then the assertion follows from Theorems
\ref{additive-noise-stability} and \ref{assumption-on-forward-map}. 
\end{proof}

Finally we turn our attention to the 
the  $G_{p,q}$-priors of Section \ref{sec:non-gaussian-priors}. 
The proof of the following corollary follows directly from Theorem
\ref{bounded-moments-product-prior} and the fact
that the $G_{p,q}$ distributions in 1D have bounded variance for $0< p,q \le 1$.
\begin{corollary}\label{well-posedness-ellp-prior}
Let $X$ be a Banach space with an unconditional Schauder basis $\{
x_k\}$ and $Y = \reals^m$. Suppose that the measurement
noise is additive and Gaussian and that the forward map $\mcl{G}$ is
bounded and linear. Then the Bayesian inverse problem
\eqref{bayes-rule} is well-posed in both the Hellinger and total
variation metrics if $\mu_0$ is an \myhl{$G_{p,q}$-prior with $0<p,q<1$.}
\end{corollary}

\section{Practical considerations and examples}\label{sec:examples}

We now turn our attention to practical aspects of solving an inverse
problem within the Bayesian framework. In the first part of this
section we discuss the problem of approximating the posterior measure 
via approximation of the likelihood potential. Afterwards, we will
present three concrete examples of Bayesian inverse problems with
heavy-tailed priors that
arise from practical problems in image deblurring and ultrasound therapy. 

\subsection{Consistent approximation of the posterior}
Up to this point we were concerned with identifying 
prior measures $\mu_0$ that result in a well-posed Bayesian inverse
problem
for a given likelihood potential $\Phi$. However, in practice we
cannot solve the inverse problem directly on the infinite-dimensional
Banach space. Therefore, we need to obtain a finite dimensional approximation
to the posterior measure $\mu^y$ which is, in some sense, consistent with the infinite
dimensional limit. 

To this end,
we will define the notion of {\it consistent approximation} of a
Bayesian inverse 
problem in the context of 
applications where one would discretize \eqref{bayes-rule} by
approximating 
the likelihood potential $\Phi$
with a discretized version $\Phi_N: X \times Y \mapsto \reals$, akin to a finite element
discretization. We define the approximation $\mu^y_N$ to 
$\mu^y$ via 
\begin{equation}\label{discretized-bayes-rule}
\frac{\dd \mu^y_N}{\dd \mu_0} = \frac{1}{Z_N(y)} \exp( - \Phi_N(u; y))
\qquad \text{where} \qquad Z_N(y) = \int_{X} \exp(- \Phi_N(u;y) )\dd \mu_0(u). 
\end{equation}

\begin{definition}[Consistent approximation{\cite{hosseini-convex-prior}}] \label{def-stability}
  The approximate Bayesian inverse problem
  \eqref{discretized-bayes-rule} is a consistent approximation to \eqref{bayes-rule} for a choice of $\mu_0$, 
$\Phi$ and $\Phi_N$ if $d( \mu^y, \mu^y_N ) \mapsto 0$ as $| \Phi(u;y) -
\Phi_N(u;y) | \mapsto 0$. Here, $d$ is either the total variation or the
Hellinger metric.
\end{definition}

This notion of a consistent approximation relates directly
to practical applications. Suppose, for example, that we are
interested in computing the expected value of a quantity $h(u)$ 
under the posterior $\mu^y$ but we can only compute the expectation
under the approximation $\mu^y_N$. If $\mu^y_N$ is a consistent
approximation in the Hellinger metric then we have, by the bound \eqref{hellinger-metric-bounds-expectation}, that if $h \in
L^2(X, \mu^y) \cap L^2(X, \mu^y_N)$ then
$$
\left| \int_X h(u) \dd \mu^y(u) - \int_X h(u) \dd \mu^y_N(u) \right|
\le C d_{H} (\mu^y, \mu^y_N).
$$ 
 In what follows, we will provide sharper bounds on the
rate of convergence of the distances between $\mu^y$ and $\mu^y_N$
under mild conditions. 
\begin{theorem}\label{stability}
  Assume that the measures $\mu^y$ and $\mu^y_N$ are defined via
  \eqref{bayes-rule} and \eqref{discretized-bayes-rule}, for a fixed
  $y \in Y$ and all values of $N$, and are absolutely continuous 
with respect to the prior $\mu_0$ which is a \myhl{Borel} probability measure
on $X$. Also 
assume that both $\Phi$ and $\Phi_N$ satisfy Assumptions
\ref{assumption-on-likelihood}(i) and (ii) with an appropriate function
$f_1$, uniformly for all $N$ and that there exists a positive and
non-decreasing function $f_3:\reals_+ \mapsto \reals_+$ so that 
\begin{equation} \label{discretization-assumption}
| \Phi(u;y) - \Phi_N(u;y) | \le f_3(\| u \|_X) \rho(N)
\end{equation}
where $\rho(N) \mapsto 0$ as $N \mapsto \infty$.
\begin{enumerate}[(i)]
\item If $
f_3(\|\cdot\|_X)  f_1(  \| \cdot \|_X) \in L^1(X, \mu_0)$ then there exists a constant $D >0$,
  independent of $N$ such that
$
d_{\text{TV}}( \mu^y , \mu^y_N) \le D \rho(N).
$
\item If $(f_3(\|\cdot\|_X))^2 f_1(  \| \cdot \|_X) \in L^1(X, \mu_0)$
  then there exists a constant $D' > 0$, independent of $N$ such that
$
d_H( \mu^y , \mu^y_N) \le D' \rho(N).
$
\end{enumerate}
\end{theorem}
\begin{proof}
Our method of proof uses similar arguments as in 
  Theorem \ref{stability-TV}
 and so we will only present it briefly for the
total variation distance. \myhl{Proof of part (ii) can also be found in \cite{stuart-bayesian-lecture-notes} for separable Banach spaces.}  The existence and uniqueness of the
measures $\mu^y$ and $\mu^y_N$ follows from Theorem
\ref{existence-uniqueness} for all values of $N$.  Next, the mean value theorem,
 Assumption \ref{assumption-on-likelihood}(i),
\eqref{discretization-assumption} and the assumption that 
$f_3(\|\cdot\|_X) f_1(\| \cdot \|_X)$ is $\mu_0$-integrable give
$$
\begin{aligned}
| Z(y) - Z_N(y) | &\le  \int_X \exp( -\Phi(u;y) ) | \Phi(u;y) - \Phi_N(u;y) | \dd \mu_0(u)  \\
&\le \left( \int_X \exp(\log(f_1( \| u \|_X) ) - M) C f_3(\| u
  \|_X) \dd \mu_0(u) \right) \rho(N)
\lesssim \rho(N).
\end{aligned}
$$
Furthermore, we have
$$
\begin{aligned}
  2 d_{TV} ( \mu^y, \mu^{y'})  
& \le \int_X \left| Z(y)^{-1} \exp( -
  \Phi(u;y) ) - Z_N(y)^{-1} \exp( - \Phi(u, y)) \right| \dd \mu_0(u) \\
& \quad + 
Z_N(y)^{-1} \int_X \left|\exp( -
  \Phi(u;y) ) - \exp( - \Phi_N(u, y)) \right| \dd \mu_0(u) =: I_1 + I_2.
\end{aligned}
$$
It then follows in a similar manner to proof of Theorem \ref{stability-TV}
 that $I_1 \lesssim \rho(N)$ 
and $I_2 \lesssim \rho(N)$ which gives the desired result.
\end{proof}

We now consider a more specific setting where
 the prior measure $\mu_0$ has
a product structure.
Suppose that the
likelihood
potential $\Phi$ satisfies the Assumption
\ref{assumption-on-likelihood} with some functions $f_1, f_2$. 
Also, assume that 
the space $X$ has an unconditional
Schauder basis $\{x_k\}$
and  consider the sequence of
spaces $(X_N, \| \cdot \|_X)$ where $X_N = \text{span} \{ x_k
\}_{k=1}^N$. These are linear subspaces of $X$ and for each $N \in
\integers$ we have $X = X_N \oplus X_N^\bot $, meaning that 
every $u \in X$ can be written as $u = u_N + u_N^\bot$ where $u_N \in
X_N$ and $u_N^\bot \in X_N^\bot $. 

Suppose that the
prior measure $\mu_0$ has the product structure of
\eqref{product-prior-sample} and assume that it has sufficiently fast
decaying tails so that the posterior measure $\mu^y$ is well-defined.  Observe that
for every value of $N$ the product prior can be factored as 
\begin{equation}\label{prior-product-decomposition}
\mu_0 = \mu_N \otimes \mu_N^\bot
\end{equation}
where $\mu_N$ and $\mu_N^\bot$ are Radon measures on $X_N$ and
$X_N^\bot$. It is natural for us to discretize the potential $\Phi$ using a
projection operator:
\begin{equation}
\label{discrete-likelihood}
\Phi_N( u;y) := \Phi(P_Nu;y)
\end{equation}
where $P_N: X \mapsto X_N$ is defined by $P_N(u) =
u_N$. Next, define the approximate posterior measures $\mu^y_N$ as
in \eqref{discretized-bayes-rule} using the above definition of
$\Phi_N$. Under these assumptions, the
$\mu^y_N$ will factor as (see
\cite[Sec.~4.1]{hosseini-convex-prior} for the details)
\begin{equation}\label{posterior-product-decomposition}
\mu^y_N = \nu_N \otimes \mu_N^\bot,
\end{equation}
where $$
\frac{\dd\nu_N}{\dd \mu_N} = \frac{1}{Z_N(y)} \exp( - \Phi(P_Nu;y)).
$$
In other words, the likelihood potential is only informative on the
subspace $X_N$ and so by comparing \eqref{prior-product-decomposition} and
\eqref{posterior-product-decomposition} we see that the approximate posterior $\mu^y_N$ differs from
the prior only on this subspace and it is identical to the prior on
$X_N^\bot$. As an example, we now check whether this method for
discretization of the posterior results in a consistent approximation
to $\mu^y$ in the additive Gaussian noise case.

\begin{theorem}\label{convergence-of-discrete-likelihood}
 Consider the above setting where the posterior and the prior have the
 prescribed 
 product structures and the $X_N$ are linear subpaces of $X$. Suppose
 that $\Phi$and 
 $\Phi_N$ are given by 
$$
\Phi(u;y) = \frac{1}{2} \| \mcl{G}(u) - y \|_2^2, \qquad \Phi_N(u;y) = \frac{1}{2} \| \mcl{G}(P_Nu) - y \|_2^2
$$
where $P_N: X \mapsto X_N$ is the projection operator that was defined
before.
Assume that the following conditions are satisfied:

\begin{enumerate}[\hspace{1ex} (a)]
\item 
$\| u - P_Nu \|_X \le \| u\|_X \rho (N).$

\item $\|\mcl{G}(u)\|_2  \le C \tilde{f}_1( \| u\|_X) \qquad \forall u \in X.$

\item 
$\| \mcl{G}(u_1) - \mcl{G}(u_2) \|_2 \le  \tilde{f}_2( \max\{
(\|u_1\|_X, \| u_2 \|_X\}) \| u_1 - u_2 \|_X \qquad \forall u_1,u_2
\in X.$ 
\end{enumerate}
Here $\rho$ is a \myhl{positive function} such that $\rho(N) \mapsto 0$ as $N \mapsto \infty$ and the functions $\tilde{f}_1, \tilde{f}_2: \reals_+ \mapsto
\reals_+$ are
non-decreasing and locally bounded and $\tilde{f}_1  \ge 1$.  Then 

\begin{enumerate}[(i)]
\item 
If $\tilde{f}_1(\| \cdot \|_X)\tilde{f}_2(\|\cdot\|_X)  \in L^1(X, \mu_0)$ then  $\exists D>0$ independent of $N$ so that 
$
d_{TV}( \mu^y, \mu^y_N) \le D \rho (N).
$
\item If $\tilde{f}_1(\| \cdot \|_X)\tilde{f}_2(\|\cdot\|_X)  \in L^2(X,
\mu_0)$ then $\exists D'>0$ independent of $N$ so that 
$
d_{H}( \mu^y, \mu^y_N) \le D' \rho (N).
$
\end{enumerate}
\end{theorem} 
\begin{proof}
  It follows from Theorem~\ref{assumption-on-forward-map} that $\Phi$
  and $\Phi_N$ satisfy Assumption \ref{assumption-on-likelihood}
  uniformly in $N$ with $M = 0$, $f_1(x) = 1$.
Then the measures $\mu^y$ and $\mu^y_N$ are well-defined for all
values of $N  \in \mbb{N}$ by
Theorem~\ref{existence-uniqueness}. Now it follows from our
assumptions on $\mcl{G}$ that
$$
\begin{aligned}
  2|\Phi(u;y) - \Phi_N(u;y)| &= \left| \| (\mcl{G}(u) - y)
  \|_{\pmb{\Sigma}}^2 - \| (\mcl{G}(P_Nu) - y)
  \|_{\pmb{\Sigma}}^2 \right|  \\ 
& = \left| \langle \pmb{\Sigma}^{-1/2} (\mcl{G}(u) - \mcl{G}(P_Nu)) ,
  \pmb{\Sigma}^{-1/2} ( \mcl{G}(u) + \mcl{G}(P_N u) - 2 y )  \rangle
\right| \\
& \le  \left(\| \mcl{G}(u) \|_{\pmb{\Sigma}} +\|
  \mcl{G}(P_Nu) \|_{\pmb{\Sigma}}
+ 2 \|y \|^2_{\pmb{\Sigma}} \right) \| (\mcl{G}(u) -
\mcl{G}(P_Nu)) \|_{\pmb{\Sigma}}  \\
&\le C \tilde{f}_1(\|u\|_X) \tilde{f}_2(\|u\|_X) \| u - P_Nu\|_X.
\end{aligned}
$$
The claim will now follow by taking $f_3(x) = 
\tilde{f}_1(x)\tilde{f}_2(x)$ and applying Theorem~\ref{stability}.
\end{proof}

A few comments are in order concerning the previous theorem. First,
the function $\rho(N)$ is independent of the forward map and the prior
and depends solely on the topology of $X$. Then the rate of
convergence of $\mu^y_N$ to $\mu^y$ depends directly on the rate of
convergence of $P_N$ to the identity map in the operator norm. Also, 
observe that in order to achieve the same rate of convergence in the
Hellinger metric as in the total variation metric, we need to impose
stronger tail assumptions on the prior $\mu_0$.

\subsection{Example 2: Deconvolution}\label{deconvolution-example}
We now turn our attention to a few concrete examples of inverse
problems with heavy-tailed or non-Gaussian prior measures.
We begin with a problem in deconvolution which is a
classic example of a linear inverse problem with wide applications in 
optics and imaging \cite{vogel, hansen-deblurring}. This problem was
also considered in \cite{hosseini-convex-prior} as an example problem
with a convex prior measure.

Let $X = L^2(\mbb{T})$ where $\mbb{T}$ is the circle of radius
$(2\pi)^{-1}$ and let $Y = \reals^m$ for a fixed integer $m$. Suppose
that $\eta \sim \mcl{N}(0, \sigma^2 \mb{I})$ where $\sigma \in \reals$
is a fixed constant and $\mb{I}$ is the $m \times m$ identity
matrix. Let $S : C(\mbb{T}) \mapsto
\reals^m$ be a bounded linear operator that collects point values of a continuous function on a set of $m$ points over $\mbb{T}$. Given a fixed kernel $\varphi \in C^1(\mathbb{T})$,
define
the forward map $\mcl{G}:X\rightarrow Y$ as 
\begin{equation}\label{deconvolution-forward-map}
\mcl{G}(u) = S( \varphi \ast u) \qquad \text{where} \qquad (\varphi\ast
u)(t) :=  \int_{\mbb{T}} \varphi(t-s) u(t) d \Lambda(s)
. 
\end{equation}
Now suppose that the data $y$ is generated via $y = \mcl{G}(u) +
\eta$ and our goal is to estimate the original image $u$ given noisy
point values of its blurred version. Note that
our assumptions so far imply a quadratic likelihood potential of the form
\eqref{gaussian-likelihood} 

It follows from Young's inequality \cite[Thm.~13.8]{heil-basis} that $(\varphi \ast \cdot):L^2(\mathbb{T}) \mapsto L^2(\mathbb{T})$ is a bounded linear 
operator and furthermore, $(\varphi \ast u) \in
C^1(\mathbb{T})$ for all $u \in
L^2(\mathbb{T})$.
Since pointwise evaluation is a bounded linear functional on
$C^1(\mathbb{T})$ then the forward map $\mcl{G}:
L^2(\mathbb{T}) \mapsto \reals^m$ is bounded and linear. We will use 
the results of Section~\ref{sec:linear-inverse-problems} to show this problem is well-posed.

We will take our prior measure to be in  class of  the product priors of Section~\ref{sec:priors-for-compressibility}.
Consider the functions 
$$
\tilde{w}(t) = 
\left\{\begin{aligned}
&  1 \quad 0\le t \le 1, \\
&0 \quad \text{otherwise}.
\end{aligned}
\right. \qquad \text{and} \qquad
\tilde{v}(t) = 
\left\{\begin{aligned}
&  1 \quad 0\le t \le 1/2, \\
&  1 \quad 1/2\le t \le 1, \\
&0 \quad \text{otherwise}.
\end{aligned}
\right.
$$
The function $\tilde{v}$ is the Haar wavelet and $\tilde{w}$ is its
corresponding scaling function. Following \cite[Sec.~9.3]{daubechies}, we can define the
periodic functions
$$
w_{jn}(t) := \sum_{l \in \mbb{Z}} \phi(2^j(t + l) - n),\qquad  v_{jn}(t) := \sum_{l \in \mbb{Z}} \psi(2^j(t + l) - n),
$$
 as well as the functions
$$
x_1(t) = w_{0,0}(t), \quad x_2(t) = v_{0,0}(t), \quad x_{2^j+n_j+1}(t) = v_{j, n_j}(t).
$$
for $j \in \mbb{Z}_+$ and $n_j = \{ 0, 1, \cdots, 2^{j}-1\}$. The $\{
x_k\}$ form an orthonormal basis for $L^2(\mbb{T})$ and so they can be
used in the construction of a $G_{p,q}$-prior.

Now choose  $p,q
\in (0,1)$ and take the prior $\mu_0$ to be the $G_{p,q}$-prior
generated by the wavelet basis $\{ x_k\}$ and the fixed sequence $\{
\gamma_k\}$  where 
$$
\gamma_{2^j + n + 1} = (1 + |2^{j+1}|^2)^{-1/2} \qquad \forall n \in
\mbb{Z}_+.
$$
 Clearly, $\{ \gamma_k\} \in \ell^2$ and so
it follows from Theorem~\ref{product-prior-ellp} that $\|
u\|_{L^2(\mbb{T})} < \infty$ a.s. Furthermore, we know that the
$G_{p,q}$-priors have bounded moments of order two. Putting this together with the fact
that the forward map $\mcl{G}$ is bounded and linear we immediately
obtain the well-posedness of this inverse problem using
 Theorem \ref{well-posedness-product-prior}.

\subsection{Example 3: Deconvolution with a BV prior}\label{example-3-BV}
\myhl{We now formulate the deconvolution problem of Example 2 with a  prior measure that is supported 
on $BV(\mbb{T})$ using the stochastic process priors of Section~\ref{sec:stochastic-process-prior}.}
\myhl{Let $u(t)$ for $t\in [0,1]$ denote a
stochastic process such that 
$$
u(0) = 0, \qquad \hat{u}_t(s)  \exp\left( t \int_\reals \exp( i \xi s)  -1 \: \dd \nu(\xi) \right)
\qquad s \in \reals.  
$$
where the measure $\nu = c \mcl{N}(0,1)$ with a fixed constant $c \in (0,\infty)$.
Then
$u$ is a compound Poisson process with piecewise constant sample
paths and normal jumps. We can write 
$
u(t) = \sum_{k=1}^{\tau(t)} \xi_k,
$
where $\{\xi_k\}$ is an i.i.d. sequence of standard normal random variables and $\tau(t)$ is a Poisson process with intensity $c$. In Section~\ref{sec:stochastic-process-prior} we saw that this process has piecewise constant sample paths and its law is a 
Radon measure on $BV([0,1])$.

Let us denote the law of this process by $\tilde{\mu}_0$. The next step is to use this measure to
define  a new measure $\mu_0$ on $BV(\mbb{T})$. 
Take $\mu_0$ to be the law of the
periodic versions
of the sample paths of the above compound Poisson process $u$ on the interval $[0,1]$.
 We can write $\mu_0 = \tilde{\mu}_0 \circ T^{-1}$
where $T: BV([0,1]) \mapsto BV(\mbb{T})$ is a bounded and linear operator.
 Thus, $\mu_0$ is a Radon measure on 
$BV(\mbb{T})$. With an abuse of notation we use $u$ to denote the
corresponding periodic
processes on $\mbb{T}$.
Since the convolution kernel $\varphi \in
C^1(\mbb{T})$ and $BV(\mbb{T}) \subset L^1(\mbb{T})$ then the forward
map $\mcl{G}: BV(\mbb{T}) \mapsto \reals^m$ (given by\eqref{deconvolution-forward-map}) is well-defined, bounded and
linear and so the likelihood potential has the
form \eqref{gaussian-likelihood} once more. We have shown, in Section~\ref{sec:stochastic-process-prior} that $\EE \| u\|_{BV(\mbb{T})} < \infty$. Putting this together with the fact that
$\mcl{G}: BV(\mbb{T}) \mapsto \reals^m$ is bounded and linear we immediately obtain the well-posedness
of this inverse problem via Corollary~\ref{linear-inverse-problems-well-posedness}.}

\subsection{Example 4: Quadratic measurements of a continuous field}
As our final example, we will consider a problem with a non-linear
forward map. Our goal is to estimate a continuous field from quadratic
measurements of its point values.  This inverse problem 
was encountered in
\cite{hosseini-hifu} in recovery of aberrations in high intensity
focused ultrasound treatment and it
 is closely
related to the phase retrieval problem \cite{dainty-phase,
  harrison-phase, candes-phaselift}.
Let $X = C(\mbb{T})$ and let $\{ t_k\}_{k=1}^n$ be a collection of
distinct points in $\mbb{T}$. Now define the operator 
$$
S: C(\mbb{T}) \mapsto \reals^n \qquad (S(u))_j= u(t_j) \qquad j
=1,2,\cdots, n. 
$$
This operator collects point values of functions in
$C(\mbb{T})$. Let $\{z_k\}_{k=1}^m$ be a fixed collection of 
vectors $z_k \in \reals^n$ and define  
the forward map 
$$
\mcl{G}: C(\mbb{T}) \mapsto \reals^m, \qquad (\mcl{G}(u))_j := |z_j^T S(u)
|^2 \qquad  \text{for} \qquad j
=1,2,\cdots, m,
$$
which collects quadratic measurements of the point values of a
continuous function. We complete our model of the measurements
with an additive layer of Gaussian noise 
$$
y = \mcl{G}(u) + \eta, \qquad \eta \sim \mcl{N}(0, \sigma^2 \mb{I}),
$$
where $\sigma > 0$. Our goal in this problem is to infer the
function $u \in C(\mbb{T})$ from the quadratic measurements $y$.

A straightforward calculation shows that 
\begin{equation}\label{quadratic-forward-map1}
\| \mcl{G}(u) \|_2 \le \tilde{K} \| S(u) \|_2^2 \le  K \| u\|_{C(\mbb{T})}^2,
\end{equation}
where $\tilde{K}, K > 0$ are constants that are independent of $u$ but
depend on the $z_k$. The last
inequality follows because pointwise evaluation is a bounded linear
operator on $C(\mbb{T})$.

Furthermore, we have that for $u_1, u_2 \in C(\mbb{T})$
$$
\begin{aligned}
(\mcl{G}(u_1)  - \mcl{G}(u_2))_j &= 
(z_j^T( S(u_1 - u_2) ))( z_j^T( S( u_1) + S(u_2)))  \\
&\le D_j ( \max\{ \| u_1 \|_{C(\mbb{T})} , \| u_2 \|_{C(\mbb{T})} \} ) \|
u_1 - u_2 \|_{C(\mbb{T})}.
\end{aligned}
$$
Here, the constant $D_j >0$ depends on $z_j$. We can now use this bound to obtain  
\begin{equation}
\label{quadratic-forward-map2}
\| \mcl{G}(u_1) - \mcl{G}(u_2) \|_2 \le D ( \max\{ \| u_1 \|_{C(\mbb{T})} , \| u_2 \|_{C(\mbb{T})} \} ) \|
u_1 - u_2 \|_{C(\mbb{T})}
\end{equation}
where the constant $D> 0$ will only depend on the $D_j$. Observe that
the above bounds in
\eqref{quadratic-forward-map1} and
\eqref{quadratic-forward-map2} imply that $\mcl{G}$ satisfies the
conditions of Theorem \ref{assumption-on-forward-map} with a function 
$\tilde{f}(x)
= x^2$. Therefore, that theorem implies that the likelihood function
$\Phi$ for our problem will satisfy Assumption
\ref{assumption-on-likelihood} (iv) with $f_2(x) = 1 + x^2$. 
Now we use Corollary \ref{additive-noise-stability} to infer that
well-posedness can be achieved if we choose a prior measure $\mu_0$ for which
$f_2(\| \cdot \|) = 1 + \|
\cdot \|_{C(\mbb{T})}^2 \in L^1(C(\mbb{T}), \mu_0)$. 

In order to construct such a prior measure $\mu_0$ we will consider a
product prior with samples of the form
$$
u \sim \sum_{k\in \mbb{Z}} \gamma_k \xi_k w_k \qquad \text{where} \qquad
w_k(t)  = (2 \pi)^{-1/2} \exp( 2 \pi i k t).
$$
The $\{w_k\}$ are simply the Fourier basis on $\mbb{T}$.
Our plan is to construct the prior measure to be supported on a
sufficiently regular Sobolev space that is embedded in
$C(\mbb{T})$. The reason for going through the Sobolev space is the 
fact that $C(\mbb{T})$ does not have an unconditional Schauder basis
and so we cannot directly apply the methodology of Section \ref{sec:priors-for-compressibility}.

To this end,
we choose 
$$
\gamma_k = ( 1 + |k|^2)^{-3/2}  \qquad k \in \mbb{Z},
$$
and suppose that the $\{\xi_k\}$ are i.i.d. and $\xi_1 \sim
\text{CPois}(0, \text{Lap}(0,1))$ (recall
Definition~\ref{def-compound-poisson}), where $\text{Lap}(0,1)$ is the
standard Laplace distribution on the real line with Lebesgue density 
$
\pi(x) = \frac{1}{2}\exp( - |x | )
$
which clearly has exponential tails
and this, in turn, implies that $\VV \xi_1 < \infty$. Note that the
random variables $\xi_k$ have a positive probability of being zero and
hence draws from this prior will incorporate a certain level of sparsity.
Observe that this is a different type of sparsity
in comparison to the $G_{p,q}$-prior. Samples from this compound
Poisson prior have a non-zero probability of having modes that are
exactly zero. The samples from the $G_{p,q}$-prior have a
zero probability of having modes that are exactly zero and instead
most of their modes will concentrate in a neighborhood of zero.

The Sobolev space $H^1(\mbb{T})$ is defined as 
$$
H^1(\mbb{T}) := \{ v \in L^2(\mbb{T}) : 
\| v \|_{H^1(\mbb{T})}^2 := \sum_{k\in \mbb{Z}} (1 + |k|^2) |\langle v, w_k \rangle|^2
 < \infty \}
$$
where $\langle \cdot, \cdot \rangle $ is the usual $L^2(\mbb{T})$ inner
product. Now consider $u \sim \mu_0$ then
$$
\| u\|_{H^1(\mbb{T})}^2 = \sum_{k \in \mbb{Z}}  ( 1 + |k|^2)^{-1}
|\xi_k|^2.
$$
But $\{ ( 1+ | k|^2 )^{-1}\} \in \ell^1$ and $\VV |\xi_k|^2 <\infty$ and
so it follows from Theorem~\ref{product-prior-ellp} that $\|
u\|_{H^1(\mbb{T})}^2 < \infty$ a.s. 
Now the Sobolev embedding
theorem \cite[Prop.~3.3]{taylor-PDE} guarantees that $\|
u\|_{C(\mbb{T})} < \infty$ a.s. and it follows from Theorem \ref{bounded-moments-product-prior} 
 that $\|
u\|_{C(\mbb{T})}^2 \in L^1(C(\mbb{T}), \mu_0)$. 

\section{Closing remarks}
\myhl{
At the beginning of this article we set out to achieve four goals. We introduced 
the new classes $G_{p,q}, W_p$ and $\ell_p$-prior measures in connection with $\ell_p$ 
regularization techniques ({\it G1}) and showed that these prior measures belong to the 
larger class of ID measures. This motivated our study of the ID class as priors ({\it G2}). 
Afterwards, we introduced another class of prior measures that were based on the laws of 
pure jump L{\'e}vy processes ({\it G3}). Our goal here was to construct a well-defined alternative to the 
classic total variation prior. 
Finally, we presented a theory of well-posedness for Bayesian inverse problems that was general enough that it 
covered the new classes of prior measures that we had introduced ({\it G4}). }
Our approach to 
well-posedness theory was to identify the minimal restrictions on the
prior measure given a choice of the likelihood potential $\Phi$.
 A common theme in our results was the trade-off between the
tail decay of the prior and the growth of the likelihood potential. As
an example, we considered the setting where the likelihood had a
quadratic form and the forward map was linear. This example corresponds to
linear inverse problems with additive Gaussian noise that are of great
interest in practice. We showed that in this simple setting
well-posedness can be achieved if the prior has moments of order one. 

Finally, we considered some practical aspects of solving inverse problems 
with heavy-tailed or ID priors. We discussed consistent discretization
of inverse problems and the use of projections in
discretization of the likelihood. Afterwards, we presented three
concrete examples of inverse problems that used heavy-tailed or
ID prior measures. In particular, we studied the well-posedness of a
deconvolution problem 
with a L{\'e}vy process prior that was cast on the non-separable space
$BV(\mbb{T})$. 

The results of this article open the door for the use of 
large classes of prior measures in inverse problems and they can be
extended in several directions.
For example, we showed that if the forward problem is linear and 
the measurement noise is Gaussian then one can achieve well-posedness
for priors that have poor tail behavior. Then many of the common
heavy-tailed priors can be used to model sparsity in the linear case. 
But it is not clear which prior is the optimal choice and in what
sense. Furthermore, given that the Compressed
Sensing literature is mainly focused on recovery of sparse signals
from linear measurements, it is interesting to study the implications
of the Compressed Sensing theory in the setting of Bayesian inverse
problems. Throughout the article we mentioned the issue of sparsity on
several occasions but this is not the only setting where non-Gaussian
priors can be useful.
For example, non-Gaussian priors can be used in modelling of
constraints or in construction of hierarchical models. 
\myhl{Finally, a major issue when it comes to using non-Gaussian priors in practice is that of sampling. 
For example, even in finite dimensions, the $G_{p,q}$-priors are far from a Gaussian measure. Then we  expect 
that Metropolis-Hastings algorithms that utilize a Gaussian proposal will have poor performance in sampling from 
posteriors that arise from $G_{p,q}$-priors. This issue will become worse as the diemension of the parameter space grows. Therefore, new sampling techniques that are tailor made to these non-Gaussian priors are needed if we wish 
to apply them in real world situations.}



 \section*{Acknowledgements}
The author owes a debt of gratitude to Prof. Nilima Nigam for many useful discussions and comments.
We are also thankful to the reviewers for their suggestions that helped
improve this manuscript significantly. 

\bibliographystyle{abbrv}
\bibliography{ref}

\begin{thebibliography}{10}

\bibitem{adler-geometry}
R.~J. Adler.
\newblock {\em The Geometry of Random Fields}.
\newblock Number~62 in Classics in Applied Mathematics. SIAM, Philadelphia,
  2010.

\bibitem{aliprantis}
C.~D. Aliprantis and K.~Border.
\newblock {\em Infinite dimensional analysis: a hitchhiker's guide}.
\newblock Springer Science \& Business Media, New York, third edition, 2006.

\bibitem{applebaum}
D.~Applebaum.
\newblock {\em {L{\'e}vy} processes and stochastic calculus}.
\newblock Number~93 in Cambridge studies in advanced mathematics. {Cambridge
  University Press, Cambridge}, 2009.

\bibitem{azais}
J.~M. Aza{\"\i}s and M.~Wschebor.
\newblock {\em Level sets and extrema of random processes and fields}.
\newblock John Wiley \& Sons, New Jersey, 2009.

\bibitem{bernardo}
J.~M. Bernardo and A.~F. Smith.
\newblock {\em Bayesian theory}.
\newblock Wiley Series in Probability and Statistics. John Wiley \& Sons, New
  York, 2009.

\bibitem{billingsley}
P.~Billingsley.
\newblock {\em Probability and measure}.
\newblock Wiley Series in Probability and Mathematical Statistics. John Wiley
  \& Sons, New York, three edition, 2008.

\bibitem{bogachev-gaussian}
V.~I. Bogachev.
\newblock {\em {Gaussian} measures}, volume~62 of {\em {Mathematical Surveys
  and Monographs}}.
\newblock {American Mathematical Society}, Providence, 1998.

\bibitem{bogachev1}
V.~I. Bogachev.
\newblock {\em Measure theory}, volume~1.
\newblock Springer, New York, 2007.

\bibitem{bogachev2}
V.~I. Bogachev.
\newblock {\em Measure theory}, volume~2.
\newblock Springer, New York, 2007.

\bibitem{bondesson}
L.~Bondesson.
\newblock A general result on infinite divisibility.
\newblock {\em The Annals of Probability}, pages 965--979, 1979.

\bibitem{borell-convex}
C.~Borell.
\newblock Convex measures on locally convex spaces.
\newblock {\em Arkiv f{\"o}r Matematik}, 12(1):239--252, 1974.

\bibitem{buttazzo}
G.~Buttazzo, M.~Giaquinta, and S.~Hildebrandt.
\newblock {\em One-dimensional variational problems: an introduction}.
\newblock Number~15 in Oxford Lecture Series in Mathematics and Its
  Applications. Oxford University Press, Oxford, 1998.

\bibitem{somersalo-inverseproblems-review}
D.~Calvetti, J.~P. Kaipio, and E.~Somersalo.
\newblock Inverse problems in the {Bayesian} framework.
\newblock {\em Inverse Problems}, 30(11):110301, 2014.

\bibitem{calvetti}
D.~Calvetti and E.~Somersalo.
\newblock {\em An introduction to {B}ayesian scientific computing: Ten lectures
  on subjective computing}, volume~2 of {\em Surveys and Tutorials in the
  Applied Mathematical Sciences}.
\newblock Springer Science \& Business Media, New York, 2007.

\bibitem{candes-phaselift}
E.~J. Candes, T.~Strohmer, and V.~Voroninski.
\newblock Phaselift: Exact and stable signal recovery from magnitude
  measurements via convex programming.
\newblock {\em Communications on Pure and Applied Mathematics},
  66(8):1241--1274, 2013.

\bibitem{scott-horseshoe}
C.~M. Carvalho, N.~G. Polson, and J.~G. Scott.
\newblock The horseshoe estimator for sparse signals.
\newblock {\em Biometrika}, 97(2):465--480, 2010.

\bibitem{vandervaart-bayesian-sparse}
I.~Castillo, J.~Schmidt-Hieber, and A.~Van~der Vaart.
\newblock Bayesian linear regression with sparse priors.
\newblock {\em The Annals of Statistics}, 43(5):1986--2018, 2015.

\bibitem{vandervaart-needles-in-haystack}
I.~Castillo and A.~van~der Vaart.
\newblock Needles and straw in a haystack: Posterior concentration for possibly
  sparse sequences.
\newblock {\em The Annals of Statistics}, 40(4):2069--2101, 2012.

\bibitem{tankov}
R.~Cont and P.~Tankov.
\newblock {\em Financial modelling with jump processes}.
\newblock Chapman \& Hall/CRC Financial mathematics series. CRC press LLC, New
  York, 2004.

\bibitem{cotter-approximation}
S.~L. Cotter, M.~Dashti, and A.~M. Stuart.
\newblock Approximation of {B}ayesian inverse problems for {PDE}s.
\newblock {\em SIAM Journal on Numerical Analysis}, 48(1):322--345, 2010.

\bibitem{dashti-besov}
M.~Dashti, S.~Harris, and A.~M. Stuart.
\newblock Besov priors for {B}ayesian inverse problems.
\newblock {\em Inverse Problems and Imaging}, 6(2):183--200, 2012.

\bibitem{stuart-bayesian-lecture-notes}
M.~Dashti and A.~M. Stuart.
\newblock {\em The Bayesian Approach to Inverse Problems}.
\newblock Springer International Publishing, 2016.

\bibitem{daubechies}
I.~Daubechies et~al.
\newblock {\em Ten lectures on wavelets}.
\newblock Number~61 in CBMS-NSF Regional Conference Series in Applied
  Mathematics. SIAM, Philadelphia, 1992.

\bibitem{evans-measure-theory}
L.~C. Evans and R.~F. Gariepy.
\newblock {\em Measure Theory and Fine Properties of Functions}.
\newblock Text Books in Mathematics. CRC Press, New York, revised edition,
  2015.

\bibitem{dainty-phase}
C.~Fienup and J.~Dainty.
\newblock Phase retrieval and image reconstruction for astronomy.
\newblock {\em Image Recovery: Theory and Application}, pages 231--275, 1987.

\bibitem{foucart}
S.~Foucart and H.~Rauhut.
\newblock {\em A mathematical introduction to compressive sensing}.
\newblock Applied and Numerical Harmonic Analysis. Springer Sience \& Business
  Media, New York, 2013.

\bibitem{ghanem}
R.~G. Ghanem and P.~D. Spanos.
\newblock {\em Stochastic finite elements: a spectral approach}.
\newblock Dover Publication Inc., New York, 2003.

\bibitem{ghosh}
P.~Ghosh and A.~Chakrabarti.
\newblock Posterior concentration properties of a general class of shrinkage
  priors around nearly black vectors.
\newblock {\em arXiv preprint at arXiv:1412.8161}, 2014.

\bibitem{hanche}
H.~Hanche-Olsen and H.~Holden.
\newblock The {K}olmogorov--{R}iesz compactness theorem.
\newblock {\em Expositiones Mathematicae}, 28(4):385--394, 2010.

\bibitem{hansen-deblurring}
P.~C. Hansen, J.~G. Nagy, and D.~P. O'leary.
\newblock {\em Deblurring images: matrices, spectra, and filtering}.
\newblock {SIAM}, Philadelphia, 2006.

\bibitem{harrison-phase}
R.~W. Harrison.
\newblock Phase problem in crystallography.
\newblock {\em JOSA A}, 10(5):1046--1055, 1993.

\bibitem{heil-basis}
C.~Heil.
\newblock {\em A basis theory primer: Expanded edition}.
\newblock Applied and Numerical Harmonic Analysis. Springer Sicence \& Business
  Media, New York, 2010.

\bibitem{hosseini-hifu}
B.~Hosseini, C.~Mougenot, S.~Pichardo, E.~Constanciel, J.~M. Drake, and J.~M.
  Stockie.
\newblock A {B}ayesian approach for energy-based estimation of acoustic
  aberrations in high intensity focused ultrasound treatment.
\newblock {\em arXiv preprint arXiv:1602.08080}, 2016.

\bibitem{hosseini-convex-prior}
B.~Hosseini and N.~Nigam.
\newblock Well-posed {B}ayesian inverse problems: priors with exponential
  tails.
\newblock 2016.
\newblock arXiv preprint at arxiv:1604.02575.

\bibitem{kotz-univariate-v1}
N.~L. Johnson, S.~Kotz, and N.~Balakrishnan.
\newblock {\em Continuous univariate distributions, Volume 1: Models and
  Applications}.
\newblock John Wiley \& Sons, New York, second edition, 2002.

\bibitem{somersalo}
J.~Kaipio and E.~Somersalo.
\newblock {\em {S}tatistical and computational inverse problems}, volume 160 of
  {\em {A}pplied {M}athematical {S}ciences}.
\newblock Springer Sience \& Business Media, New York, 2005.

\bibitem{kotz-laplace}
S.~Kotz, T.~J. Kozubowski, and K.~Podgorski.
\newblock {\em The Laplace distribution and generalizations: A revisit with
  applications to communications, economics, engineering, and finance}.
\newblock Springer Science \& Business Media, New York, 2012.

\bibitem{kruglov}
V.~Kruglov.
\newblock A note on infinitely divisible distributions.
\newblock {\em Theory of Probability \& Its Applications}, 15(2):319--324,
  1970.

\bibitem{lassas-can-we-use-tv}
M.~Lassas and S.~Siltanen.
\newblock Can one use total variation prior for edge-preserving {B}ayesian
  inversion?
\newblock {\em Inverse Problems}, 20(5):1537, 2004.

\bibitem{leoni}
G.~Leoni.
\newblock {\em A First Course in Sobolev spaces}, volume 105 of {\em Graduate
  Studies in Mathematics}.
\newblock 2009.

\bibitem{linde}
W.~Linde.
\newblock {\em Probability in {B}anach spaces: Stable and infinitely divisible
  distributions}.
\newblock John Wiley \& Sons, New York, 1986.

\bibitem{lucka-dissertation}
F.~Lucka.
\newblock {\em Bayesian inversion in biomedical imaging}.
\newblock PhD thesis, University of Muenster, december 2014.

\bibitem{nadarajah-kotz}
S.~Nadarajah.
\newblock The kotz-type distribution with applications.
\newblock {\em Statistics: A Journal of Theoretical and Applied Statistics},
  37(4):341--358, 2003.

\bibitem{nadarajah-generalized-normal}
S.~Nadarajah.
\newblock A generalized normal distribution.
\newblock {\em Journal of Applied Statistics}, 32(7):685--694, 2005.

\bibitem{paciorek-dissertation}
C.~J. Paciorek.
\newblock {\em Nonstationary Gaussian processes for regression and spatial
  modelling}.
\newblock PhD thesis, Carnegie Mellon University, May 2003.

\bibitem{peszat}
S.~Peszat and J.~Zabczyk.
\newblock {\em Stochastic partial differential equations with L{\'e}vy noise:
  An evolution equation approach}, volume 113 of {\em Encyclopedia of
  Mathematics and its Applications}.
\newblock Cambridge University Press, Cambridge, 2007.

\bibitem{scott-shrink}
N.~G. Polson and J.~G. Scott.
\newblock Shrink globally, act locally: Sparse {B}ayesian regularization and
  prediction.
\newblock {\em {B}ayesian Statistics}, 9:501--538, 2010.

\bibitem{rasmussen}
C.~E. Rasmussen and W.~C.~K. I.
\newblock {\em {G}aussian processes for machine learning}.
\newblock the MIT press, Cambridge, 2006.

\bibitem{sato}
K.-i. Sato.
\newblock {\em L{\'e}vy processes and infinitely divisible distributions}.
\newblock Number~68 in Cambridge studies in advanced mathematics. Cambridge
  university press, Cambridge, 1999.

\bibitem{steutel}
F.~W. Steutel and K.~Van~Harn.
\newblock {\em Infinite divisibility of probability distributions on the real
  line}.
\newblock Pure and Applied Mathematics. Marcel Dekker Inc., New York, 2003.

\bibitem{stuart-acta-numerica}
A.~M. Stuart.
\newblock Inverse problems: a {B}ayesian perspective.
\newblock {\em Acta Numerica}, 19:451--559, 2010.

\bibitem{sullivan}
T.~J. Sullivan.
\newblock Well-posed bayesian inverse problems and heavy-tailed stable {B}anach
  space priors.
\newblock 2016.
\newblock arXiv preprint at arxiv:1605.05898.

\bibitem{taylor-PDE}
M.~E. Taylor.
\newblock {\em Partial Differential Equations I: Basic Theory}, volume 115 of
  {\em Applied Mathematical Sciences}.
\newblock Springer Science \& Business Media, New York, second edition, 2011.

\bibitem{unser}
M.~Unser and P.~Tafti.
\newblock {\em An introduction to sparse stochastic processes}.
\newblock {Cambridge University Press, Cambridge}, 2013.

\bibitem{unser-unified2}
M.~Unser, P.~Tafti, A.~Amini, and H.~Kirshner.
\newblock A unified formulation of {G}aussian vs. sparse stochastic processes.
  {P}art {II}: {D}iscrete-domain theory.
\newblock {\em IEEE Transactions on Information Theory}, 60:3036--3051, 2011.

\bibitem{unser-unified}
M.~Unser, P.~Tafti, and Q.~Sun.
\newblock A unified formulation of {G}aussian vs. sparse stochastic processes.
  {P}art {I}: {C}ontinuous-domain theory.
\newblock {\em IEEE Transactions on Information Theory}, 60:1945--1962, 2011.

\bibitem{vogel}
C.~R. Vogel.
\newblock {\em {C}omputational methods for inverse problems}.
\newblock {SIAM, Philadelphia}, 2002.

\bibitem{yao-tv-gaussian}
Z.~Yao, Z.~Hu, and J.~Li.
\newblock A {TV-Gaussian} prior for infinite-dimensional {Bayesian} inverse
  problems and its numerical implementations.
\newblock {\em Inverse Problems}, 32(7):075006, 2016.

\end{thebibliography}


\end{document}